\documentclass[reqno,11pt]{amsart}
\usepackage{latexsym,amssymb,amsfonts,amsmath,amsthm}
\usepackage[usenames]{xcolor}
\usepackage{eucal}

\usepackage{mathtools}
\usepackage{mathabx}

\usepackage[colorlinks,citecolor=blue,urlcolor=blue]{hyperref}

\newtheorem{thm}{Theorem}[section] 
\newtheorem{cor}[thm]{Corollary}
\newtheorem{lem}[thm]{Lemma}
\newtheorem{prop}[thm]{Proposition}
\newtheorem{defn}[thm]{Definition}

\theoremstyle{remark}
\newtheorem{rem}[thm]{Remark}
\numberwithin{equation}{section}
\newcommand{\Real}{\mathbb R}
\newcommand{\eps}{\varepsilon}

\newcommand{\F}{\mathcal{F}}

\newcommand{\one}[1]{\mathbf{1}_{\{#1\}}}

\renewcommand{\P}{\mathbb{P}}

\newcommand{\E}{\mathbb{E}}

\DeclareMathOperator*{\sign}{sign}

\DeclareMathOperator{\cov}{cov}

\renewcommand{\Im}{\mathrm{Im}}
\renewcommand{\i}{\mathrm{i}}

\def\Xint#1{\mathchoice
   {\XXint\displaystyle\textstyle{#1}}%
   {\XXint\textstyle\scriptstyle{#1}}%
   {\XXint\scriptstyle\scriptscriptstyle{#1}}%
   {\XXint\scriptscriptstyle\scriptscriptstyle{#1}}%
   \!\int}
\def\XXint#1#2#3{{\setbox0=\hbox{$#1{#2#3}{\int}$}
     \vcenter{\hbox{$#2#3$}}\kern-.5\wd0}}

\def\dashint{\Xint-}

\makeatletter
\newcommand{\doublewidetilde}[1]{{%
  \mathpalette\double@widetilde{#1}%
}}
\newcommand{\double@widetilde}[2]{%
  \sbox\z@{$\m@th#1\widetilde{#2}$}%
  \ht\z@=.9\ht\z@
  \widetilde{\box\z@}%
}
\makeatother

\begin{document}

\title[]
{
Estimation of the Hurst parameter from continuous noisy data
}

\author{P. Chigansky}%
\address{Department of Statistics,
The Hebrew University,
Mount Scopus, Jerusalem 91905,
Israel}
\email{pchiga@mscc.huji.ac.il}

\author{M. Kleptsyna}%
\address{Laboratoire de Statistique et Processus,
Le Mans Universit\'{e},
France}
\email{marina.kleptsyna@univ-lemans.fr}

\thanks{P.Chigansky is supported by ISF 1383/18 grant; M.Kleptsyna is supported by ANR EFFI grant}



\date{\today}%
\begin{abstract}
This paper addresses the problem of estimating the Hurst exponent of the fractional Brownian motion from 
continuous time noisy sample. Consistent estimation in the setup under consideration is possible only if either the length of the 
observation interval increases to infinity or intensity of the noise decreases to zero. The main result is a proof of the Local 
Asymptotic Normality (LAN) of the model in these two regimes, which reveals the optimal minimax rates.
\end{abstract}
\maketitle

\section{Introduction}

Estimation of the Hurst parameter is an old problem in statistics of time series. 
A benchmark model is the fractional Brownian motion (fBm) $B^H = (B^H_t, \ t\in \Real_+)$, i.e., 
the centered Gaussian process with covariance function
$$
\E B^H_t   B^H_s = \tfrac 1 2 t^{2H} + \tfrac 1 2 s^{2H} - \tfrac 1 2 |t-s|^{2H},
$$
where $H\in (0,1)$ is the Hurst exponent.
The fBm is a well studied stochastic process with a variety of interesting and useful properties, see, e.g., 
\cite{PT17}, \cite{EM02}. Its increments are stationary and, for $H>\frac 1 2$, positively correlated with 
long-range dependence 
$$
\sum_{n=1}^\infty \E B^H_1 (B^H_{n}-B^H_{n-1}) =\infty.
$$
It is this feature which makes the fBm relevant to statistical modeling in many applications.  

A basic problem is to estimate the Hurst parameter $H\in (0,1)$ and the additional scaling parameter $\sigma\in \Real_+$
given the data 
$$
X^T := (\sigma B^H_t, \, t\in [0,T]).
$$ 
Since both parameters can be recovered from $X^T$ exactly for any $T>0$, 
a meaningful statistical problem is to estimate them from the discretized data 
\begin{equation}\label{BDelta}
X^{T,\Delta} := \big(\sigma B^H_{\Delta},\, ...,\, \sigma B^H_{n\Delta}\big)
\end{equation}
where $\Delta>0$ is the discretization step and $n=[T/\Delta]$. The two relevant regimes, in which consistent estimation 
from \eqref{BDelta} is feasible, are the large time asymptotics with a fixed $\Delta >0$ and $T\to\infty$, 
and the high frequency asymptotics  with $\Delta \to 0$ and a fixed $T>0$.   

A statistically harder problem is obtained if an independent noise is added to the sample. One possibility, which was explored so far,
is estimation from the discrete noisy data
\begin{equation}\label{BDeltaxi}
X^{T,\Delta} := \big(\sigma B^H_{\Delta }+\xi_{1,n},\, ...,\, \sigma B^H_{n\Delta}+\xi_{n,n}\big),
\end{equation}
where $\xi_{j,n}$ are i.i.d. random variables independent of $B^H$.
While presence of noise does not alter the large time asymptotics in any drastic way, it does slow down the optimal minimax rates
in the high frequency regime, see Section \ref{sec:dis} for more details.

In this paper we consider an alternative scenario with {\em continuous time} noisy sample obtained by adding to the fBm 
an independent standard Brownian motion $B$:
\begin{equation}\label{mfBm}
X^T := \big(\sigma B^H_t + \sqrt \eps B_t, t\in [0,T]\big),
\end{equation}
where $\eps>0$ is the known noise intensity. From the statistical standpoint, a peculiar feature of this process, called 
the {\em mixed fBm}, is that consistent estimation in the high frequency regime is no longer possible.
Indeed, it was shown in \cite{Ch01} that for $H>\frac 34$ the probability measures induced by the the process \eqref{mfBm} 
and the Brownian motion $\sqrt \eps B$ are mutually absolutely continuous.
This implies that the parameters in question cannot be recovered exactly from the sample $X^T$ with any finite $T$ and, a fortiori, 
from its discretization.

Our contribution is a proof of the local asymptotic normality (LAN) property in the large time ($T\to\infty$ and fixed $\eps>0$)
and small noise ($T<\infty$ fixed and $\eps\to 0$) asymptotic regimes. In both cases the analysis reveals the optimal minimax 
rates and yields explicit expressions for the relevant Fisher information matrices. The rest of the paper is organized as follows.
Section \ref{sec:bac} outlines the essential background needed for formulation of the main results in Section 
\ref{sec:main}. The results are discussed and compared to the related literature in Section \ref{sec:dis}.
The proofs appear in Sections \ref{sec:prev}-\ref{sec:sn}.

\section{The LAN property}\label{sec:bac}

Let us briefly recall Le Cam's LAN property and its role in the asymptotic theory of estimation. A comprehensive account on the 
subject can be found in, e.g., \cite{IKh}. An abstract parametric statistical experiment consists of a measurable 
space $(\mathcal X, \mathcal A)$, where $\mathcal A$ is a $\sigma$-algebra of subsets of $\mathcal X$,
a family of probability measures $(\P_\theta)_{\theta\in \Theta}$ on $\mathcal A$ with the parameter space
$\Theta\subseteq \Real^k$ and the sample $X\sim \P_{\theta_0}$ for a true value $\theta_0\in \Theta$ of the parameter variable. 
Asymptotic theory is concerned with a family of statistical experiments $(\mathcal X^h, \mathcal A^h, (\P^h_\theta)_{\theta\in \Theta})$  
indexed by a real valued variable $h>0$.

\begin{defn}\label{defn:LAN}
A family of probability measures $(\P^h_\theta)_{\theta\in \Theta}$ is locally asymptotically normal (LAN) at a point $\theta_0$
as $h\to 0$ if there exist nonsingular $k\times k$ matrices $\phi(h)=\phi(h, \theta_0)$, such that for any $u\in \Real^k$, 
the Radon-Nikodym derivatives (likelihood ratios) satisfy the scaling property 
\begin{equation}\label{LR}
\log \frac{d\P^h_{\theta_0+\phi(h)u}}{d\P^h_{\theta_0}}(X^h) = u^\top Z_{h,\theta_0} - \tfrac 1 2 \|u\|^2
+ r_h(u,\theta_0),
\end{equation}
where the random vector $Z_{h, \theta_0}$ converges weakly under $\P^h_{\theta_0}$ to the standard normal law on $\Real^k$
and $r_h (u,\theta_0)$ vanishes in $\P^h_{\theta_0}$-probability as $h\to 0$.
\end{defn}

Define a set $W_{2,k}$ of loss functions $\ell:\Real^k\mapsto \Real_+$, which are continuous and symmetric 
with $\ell(0)=0$, have convex sub-level sets $\{u:\ell(u)<c\}$ for all $c>0$ and satisfy the growth condition 
$\lim_{\|u\|\to 0}\exp(-a \|u\|^2)\ell(u)=0$, $\forall a>0$.
The following theorem establishes asymptotic lower bound for the corresponding local minimax risks of estimators in LAN families.

\begin{thm}[H\'ajek]\label{thm:HLC}
Let $(\P^h_\theta)_{\theta \in \Theta}$ satisfy the LAN property at $\theta_0$ with matrices $\phi(h, \theta_0)\to 0$ as $h\to 0$.  
Then for any family of estimators $\widehat \theta_h$, a loss function $\ell\in W_{2,k}$ and any $\delta>0$,
$$
\varliminf_{h\to 0}\sup_{\|\theta-\theta_0\|<\delta} \E^h_\theta \ell \big(\phi(h,\theta_0)^{-1} (\widehat \theta_h-\theta)\big)
\ge \int_{\Real^k} \ell(x) \gamma_k(x)dx,
$$
where $\gamma_k$ is the standard normal density on $\Real^k$.
\end{thm}

\begin{proof} 
 \cite[Theorem 12.1]{IKh}.
\end{proof}

Estimators which achieve H\'ajek's lower bound are called asymptotically efficient in the local minimax sense. 
Often likelihood based estimators, such as the Maximum Likelihood or the Bayes estimators, are asymptotically efficient. 
However, they can be excessively complicated and it is then desirable to construct simpler estimators, which are at least 
rate optimal. In complex models this can be easier to carry out separately for each component of the parameter vector, following some ad-hoc 
heuristics. Proving rate optimality of so obtained estimators requires finding the best minimax rates for each
parameter.

Let us explain how such entrywise rates can be derived using the bound from Theorem \ref{thm:HLC}. 
Analysis of the likelihood ratio in \eqref{LR} typically shows that in LAN families $\phi(h, \theta_0)$ must satisfy the condition 
\begin{equation}\label{cond}
\phi(h, \theta_0)^\top M(h,\theta_0) I(\theta_0) M(h,\theta_0)^\top \phi(h, \theta_0)\xrightarrow[h\to 0]{}\mathrm{Id},
\end{equation}
where the matrices $M(h,\theta_0)$ and $I(\theta_0)$ are determined by the statistical model under consideration. 
The matrix $I(\theta_0)$ is positive definite and independent of $h$, and it can be often regarded as the analog of the 
usual Fisher information matrix.

Consider the Cholesky decomposition 
$$
L(h,\theta_0)L(h,\theta_0)^\top = M(h,\theta_0) I(\theta_0) M(h,\theta_0)^\top,
$$
where $L(h,\theta_0)$ is the unique lower triangular matrix with positive diagonal entries.
Then \eqref{cond} holds with $\phi(h, \theta_0)^\top :=L(h,\theta_0)^{-1}$ and the last entry of the vector 
$\phi(h,\theta_0)^{-1} (\widehat \theta_h-\theta)$ is given by
$$
\big[\phi(h,\theta_0)^{-1} (\widehat \theta_h-\theta)\big]_{k} = \big[L(h,\theta_0)^\top  (\widehat \theta_h-\theta)\big]_{k}=
L_{kk}(h,\theta_0) (\widehat \theta_{h,k}-\theta_k).
$$
Let $\widetilde \ell \in W_{2,1}$ be a loss function of a scalar  variable, $\widetilde \ell:\Real \mapsto \Real_+$, 
and define $\ell(x) := \widetilde \ell(x_k)$, $x\in \Real^k$. This loss function  belongs to $W_{2,k}$ and H\'ajek's bound implies 
\begin{equation}\label{sbnd}
\varliminf_{h\to 0}\sup_{\|\theta-\theta_0\|<\delta} \E^h_\theta \widetilde \ell \big( L_{kk}(h,\theta_0) (\widehat \theta_{h,k}-\theta_k)\big)
\ge \int_{\Real} \widetilde \ell(t) \gamma_1(t)dt >0.
\end{equation}
This inequality identifies the last diagonal entry of $L(h,\theta_0)$ as the best minimax rate in estimation of $\theta_k$.
Similar bound for an arbitrary entry is obtained by permuting the components of $\theta$ so that it becomes the last.

The most commonly encountered instance of \eqref{cond} is when the matrix $M(h,\theta_0)$ is diagonal.  
Then  $L(h,\theta_0) = M(h,\theta_0) S(\theta_0)$ where $S(\theta_0)$ is the Cholseky factor of $I(\theta_0)$. 
Since $I(\theta_0)$ is positive definite, all diagonal entries of $S(\theta_0)$ are positive (and constant in $h$) 
and, in view of \eqref{sbnd}, the best minimax rate is determined only by $M_{kk}(h,\theta_0)$.
This is the case for our model in the large time asymptotic regime with $h:=1/T$ (see Theorem \ref{thm:main}).
In the small noise regime with $h:=\eps$, the matrix $M(h,\theta_0)$ is non-diagonal (see Theorem \ref{thm2}),
which results in a logarithmic discrepancy between the best minimax rates in estimation of each parameter.

\section{Main results}\label{sec:main}

\subsection{Large time asymptotics} 
The covariance function of the fBm with parameter variable 
$
\theta = (H,\sigma^2) \in (\tfrac 3 4 , 1) \times (0,\infty) =: \Theta
$
can be written as 
$$
\cov (B_s^H, B_t^H) = \int_0^s \int_0^t K_\theta(u-v) dudv
$$
where 
\begin{equation}\label{KH}
K_\theta(\tau) =  \sigma^2 H (2H-1) |\tau|^{2H-2}.
\end{equation}
The Fourier transform of this kernel has the explicit formula 
\begin{equation}\label{KK}
\widehat K_\theta(\lambda) = \int_\Real K_\theta(\tau)e^{-\i\lambda \tau}d\tau= \sigma^2 a_H |\lambda|^{1-2H}, \quad \lambda \in \Real \setminus\{0\},
\end{equation}
with the constant $a_H:= \Gamma(2H+1)\sin (\pi H)$.
The function $\widehat K_\theta(\lambda)$ does not decay sufficiently fast to be integrable on $\Real$ and hence, 
strictly speaking, it is not a spectral density of a stochastic process in the usual sense.
It can be thought of as the spectral density of the {\em fractional noise}, a formal derivative of the fBm.
  
Denote by $\P^T_\theta$ the probability measure on the space of continuous functions $C([0,T],\Real)$, induced by the mixed fBm  
\eqref{mfBm} with parameter $\theta$ and $\eps>0$ being fixed.

\begin{thm}\label{thm:main}
The family $(\P^T_\theta)_{\theta\in \Theta}$ is LAN at any $\theta_0\in \Theta$ as $T\to\infty$ with  
$$
\phi(T) =T^{-1/2} I(\theta_0, \eps)^{-1/2},
$$
where $I(\theta_0, \eps)$ is the Fisher information matrix 
\begin{equation}\label{Ieps}
I(\theta,\eps)  = \frac 1 {4\pi} \int_{-\infty}^\infty \nabla^\top \log \big(\eps +  \widehat K_\theta(\lambda)\big)
\nabla \log \big(\eps +  \widehat K_\theta(\lambda)\big) d\lambda>0,
\end{equation}
and $\nabla$ denotes the gradient with respect to parameter variable $\theta$.
\end{thm}

\medskip

In view of the discussion in the previous section, this result implies that the rate $T^{-1/2}$ is minimax optimal
for both $H$ and $\sigma^2$.  As explained in Section \ref{subsec:1}, 
this rate is achievable and, moreover, the lower bound can be approached arbitrarily close by estimators based on 
sufficiently dense grid of discretized observations.

\subsection{Small noise asymptotics}
With a convenient abuse of notations, let $\P^\eps_\theta$  now denote the probability measure induced by the 
mixed fBm \eqref{mfBm} with parameter $\theta$ and a fixed interval length $T>0$. Define the  matrix 
$$
M(\eps, \theta) = \eps^{-1/(4H-2)}\begin{pmatrix}
1 & -2\sigma^2  \log\eps^{-1/(2H-1)}\\
0 & 1
\end{pmatrix}.
$$
\begin{thm}
\label{thm2}  
Assume that $\phi(\eps, \theta_0)$ satisfies the scaling condition 
\begin{equation}\label{condphi}
\phi(\eps,\theta_0)^\top M(\eps,\theta_0) T I(\theta_0,1)M(\eps,\theta_0)^\top\phi(\eps,\theta_0)\xrightarrow[\eps\to 0]{}\mathrm{Id},
\end{equation}
with $I(\theta_0,1)$  defined in \eqref{Ieps}.
Then the family $(\P^\eps_\theta)_{\theta \in \Theta}$ is LAN at any $\theta_0\in \Theta$ as $\eps\to 0$.  
\end{thm}

Condition \eqref{condphi} cannot be satisfied by any diagonal matrix $\phi(\eps,\theta_0)$, since in this case the limit, if exists and finite, 
must be a singular matrix. Otherwise the choice of $\phi(\eps,\theta_0)$ is not unique. As explained in the previous section, 
the upper and lower triangular Cholesky factors of the matrix $M(\eps,\theta_0) I(\theta_0,1)M(\eps,\theta_0)^\top$ reveal the optimal 
minimax estimation rates for $H$ and $\sigma^2$.

\begin{cor}\label{cor1}\

\medskip
\noindent
1) For any family of estimators $\widehat H_\eps$, a loss function $\ell\in W_{2,1}$ on $\Real$ and $\delta>0$,
$$
\varliminf_{\eps\to 0}\sup_{\|\theta-\theta_0\|<\delta} \E_{\theta} \ell \big( \eps^{-1/(4H_0-2)}  (\widehat H_\eps-H)\big)
\ge \int_{\Real} \ell(x/J(\theta_0)) \gamma(x)dx,
$$
where $\gamma$ is the standard normal density on $\Real$ and 
$$
J(\theta): =\sqrt{T \Big(I_{11}(\theta,1)-\frac{I_{12}(\theta,1)^2}{I_{22}(\theta,1)}\Big)}.
$$
 
\medskip 
\noindent
2) For any family of estimators $\widehat \sigma^2_\eps$, a loss function $\ell \in W_{2,1}$ on $\Real$ and 
$\delta>0$,
$$
\varliminf_{\eps\to 0}\sup_{\|\theta-\theta_0\|<\delta} \E_{\theta} \ell \Big( \eps^{-1/(4H_0-2)} \frac 1{\log\eps^{-1}} 
\big(\widehat \sigma^2_\eps-\sigma^2\big)\Big)\ge \int_{\Real} \ell(x/J(\theta_0)) \gamma(x)dx,
$$
where $\gamma$ is the standard normal density on $\Real$ and 
$$
J(\theta) :=\frac{H -\frac 1 2}{\sigma^2}\sqrt{T \Big(I_{11}(\theta,1)-\frac{I_{12}(\theta,1)^2}{I_{22}(\theta,1)}\Big)}.
$$

\end{cor}

If only one parameter is to be estimated, while the other one is known, the relevant LAN property corresponds to the respective
one-dimensional family. The following theorem shows that the optimal minimax rates is these cases improve by a logarithmic factor.

\goodbreak

\begin{thm}\label{thm3}\

\medskip
\noindent
1) For any fixed $\sigma_0^2\in \Real_+$, the family $\Big(\P^\eps_{(H,\sigma^2_0)}\Big)_{H\in (\frac 34,1)}$ is LAN at any $H_0\in (\frac 3 4 ,1)$
as $\eps\to 0$ with 
\begin{equation}\label{rateH}
\phi(\eps,H_0) := \eps^{1/(4H_0-2)}\frac 1 {\log \eps^{-1}} \frac{H_0-\frac 1 2}{ \sigma^2_0}\frac 1 {\sqrt{TI_{22}(\theta_0,1)}}.
\end{equation}

\medskip
\noindent
2) For any fixed $H_0\in (\frac 34,1)$, the family $\Big(\P^\eps_{(H_0,\sigma^2)}\Big)_{\sigma^2\in (0,\infty)}$ is LAN at any $\sigma^2_0\in \Real_+$ as $\eps\to 0$ with 
$$
\phi(\eps,\sigma_0) := \eps^{1/(4H_0-2)}\frac 1 {\sqrt{TI_{22}(\theta_0,1)}}.
$$
\end{thm}

\section{A discussion}\label{sec:dis}

\subsection{On the information matrix}

The expression for the Fisher information matrix in \eqref{Ieps} is known as Whittle's formula. 
It was discovered by P.\,Whittle \cite{W53, W62} and was originally derived for discrete time stationary 
Gaussian processes with continuous spectral densities, see also \cite{Walker64}. 
Its validity was extended in \cite{D89, D89c} to processes with long range dependence, for which the spectral density has
an integrable singularity at the origin.

Whittle's formula in continuous time is a more subtle matter due to complexity of the absolute continuity relation
between Gaussian measures on function spaces.
In fact, according to the survey \cite{DY83}, it was never rigorously verified beyond processes with rational spectra.
One important class for which further generalization is plausible are processes observed with additive ``white noise'', that is, 
$$
X_t = Z_t + B_t, \quad t\in [0,T],
$$
where $B$ is a standard Brownian motion and $Z$ is a centered Gaussian process with stationary increments.
The mixed fBm is a special case from this class. 

Results in \cite{Sh66} imply that the probability measure  induced by $X$ is equivalent to the Wiener measure 
if and only if 
$$
\E Z_t Z_s = \int_0^t \int_0^s K_\theta (u-v)dudv 
$$
for some kernel $K_\theta\in L^2([0,T])$. In this case, the Radon-Nikodym derivative  has the same form as in \eqref{RN}.
Using the theory of finite section approximation from \cite{GF74} it is indeed possible to prove Theorem \ref{thm:main} 
for such processes under the additional, crucial to the approach of \cite{GF74}, assumption $K\in L^1(\Real)$. 

This assumption is violated by the kernel \eqref{KH}, which makes the method of \cite{GF74} inapplicable.  
This is not entirely surprising in view of the difficulties, 
needed to be overcome in \cite{D89} to extend Whittle's theory to discrete time processes with the long range dependence. 
The results in our paper are proved using a different approach, based on the ideas from \cite{Ukai} and 
the recent applications to processes with the fractional covariance structure \cite{ChK}.

\subsection{On the joint and separate estimation}
Logarithmic discrepancy in the minimax rates between joint and separate estimation as in Corollary \ref{cor1} and Theorem \ref{thm3} 
is known to occur in the high-frequency regime in experiments with discrete data such as \eqref{BDelta}. 
The optimal rates for the separate estimation of $H$ and $\sigma^2$ for  $\Delta = T/n$ are 
$$
n^{-1/2}\frac  1{\log n}\quad \text{and}\quad n^{-1/2}
$$
respectively, see \cite{K13} and references therein.
These rates are achievable, e.g., by estimators based on discrete power variations as in \cite{IL97}, \cite{KW97}, \cite{C01}. 

It was long noticed that analogous estimators achieve slower rates, degraded by logarithmic factor,
\begin{equation}\label{rjoint}
n^{-1/2}\quad \text{and}\quad n^{-1/2}\log n,
\end{equation} 
when both parameters are unknown. 
These rates were recently proved minimax optimal in \cite{BF17}, 
where the LAN property was shown to hold with a non-diagonal matrix $M(h,\theta_0)$ in \eqref{cond}.

High frequency estimation from the noisy data \eqref{BDeltaxi} was considered in \cite{GH07}, 
where the optimal minimax rates for joint estimation of $H>\frac 12 $ and $\sigma^2$ were found to be 
$$
n^{-1/(4H+2) }\quad \text{and}\quad  n^{-1/(4H+2)}\log n,
$$ 
respectively. These rates are slower than those in \eqref{rjoint}, confirming the intuition that noise should make the 
estimation problem harder.
Estimators for the mixed fBm with $H\le \frac 34$ in the high frequency regime were constructed using the power variations 
technique in \cite{DMS14}.

\subsection{Rate optimal estimators}\label{subsec:1}

It is plausible that the local minimax lower bounds guaranteed by the LAN property of Theorems \ref{thm:main} - \ref{thm3}
are attained by the Maximum Likelihood or the Bayes estimators with positive prior densities. Proving such asymptotic efficiency 
would require analysis which goes beyond the scope of this paper.  
However these estimators are of little practical interest, since they require solving numerically the integral 
equation \eqref{eq} and approximating the stochastic integrals in \eqref{RN}. Simpler rate optimal estimators can be constructed 
otherwise as explained below. 

\subsubsection{Large time}

A simpler alternative is to base estimation on the discrete data $X_{k\Delta}$, $k=1,...,[T/\Delta]$ 
with a small discretization step $\Delta>0$. In particular, one can use the discrete likelihood based estimators 
or the simpler Whittle's spectral estimator. The theory from \cite{D89} implies that such estimators achieve the optimal 
$T^{-1/2}$ rate and, moreover, the limit risks can be made arbitrarily close to the bound provided by Theorem \ref{thm:main} 
by choosing $\Delta$ small enough.
 
\subsubsection{Small noise}
 
In the small noise regime the optimal rates of Corollary \ref{cor1} and Theorem \ref{thm3}
can be achieved by a modification of the estimator suggested in \cite{GH07}. Let us briefly sketch the 
idea. Take any mother wavelet function $\psi$ with compact support and two vanishing moments. Define
its translates and dilations   
$$
\psi_{j,k}(t) = 2^{j/2} \psi(2^j t-k), \quad j\in \mathbb N,\ k\in \mathbb Z. 
$$
Consider the wavelet coefficients of $\sigma B^H$  
\begin{equation}\label{djk}
d_{j,k} = \int_\Real \psi_{j,k}(t)\sigma dB^H_t
\end{equation}
and define the energy of the $j$-th resolution level
$$
Q_j = \sum_{k=0}^{2^{j-1}-1} d_{j,k}^2.
$$
Standard calculations show that these random variables satisfy  
\begin{equation}\label{dva}
Q_j = \frac {\sigma^2} 2 c_H(\psi)  2^{  j (2-2H)} + O_\P\big(2^{-j(4H-3)/2  }\big),\quad j\to \infty,
\end{equation}
where 
$$
c_H(\psi) = \int_{\Real}\int_{\Real} \psi(u)\psi(v)H(2H-1)|u-v|^{2H-2}dudv.
$$
Consequently,
\begin{equation}\label{raz}
\frac{Q_{j+1}}{Q_j}= 2^{2-2H} + O_\P(2^{-j/2}),  \quad j\to \infty.
\end{equation}

Natural estimators for the wavelet coefficients are obtained by replacing the fBm in \eqref{djk} with its noisy observation 
$$
\widetilde d_{j,k} := \int_\Real \psi_{j,k}(t)  dX_t.
$$
Since $\E \widetilde d_{j,k}^2 = \E d_{j,k}^2 + \eps \|\psi\|^2$ it makes sense to estimate $d_{j,k}^2$ by
$$
\widehat {d_{j,k}^2} = (\widetilde d_{j,k})^2-\eps \|\psi\|^2
$$
and, accordingly,  
$$
\widehat Q_j =\sum_{k=0}^{2^{j-1}-1} \widehat {d_{j,k}^2}.
$$
In view of \eqref{raz}, the method of moments suggests the estimators 
$$
\widehat H_j = 1-\frac 1 2 \log_2 \frac{\widehat Q_{j+1}}{\widehat Q_j}.
$$ 
The bias of these estimators decreases with $j$ whereas their variance increases. 
In view of the residual in \eqref{raz} and the optimal rate, known from Corollary \ref{cor1}, 
it is reasonable to suggest that the optimal choice of $j$ should be such that
\begin{equation}\label{oracle}
2^{-j/2} = \eps^{1/ (4H-2)}.
\end{equation}
This choice is only an ``oracle'' since it requires $H$ to be known. To mimic this choice of $j$, asymptotics \eqref{dva} can be used again. 
To this end \eqref{oracle} can be rewritten as $2^{ j (2-2H)}= 2^j \eps$, which, in view of \eqref{dva}, suggests the selector
$$
J^*_\eps = \max\big\{\underline J \le j\le J_\eps: \widehat Q_j \ge 2^j \eps\big\},
$$
where $J_\eps=[2 \log_2\eps^{-1}]$ and $\underline J$ is an arbitrary nonessential constant. 
It can be shown that with high probability $J^*_\eps$ will be close to $\frac 1 {2H-1}\log_2 \eps^{-1}< J_\eps$ and  
the ultimate estimator is set to be 
\begin{equation}\label{Hhat}
\widehat H(\eps) := \widehat H_{J^*_\eps}.
\end{equation}

\begin{prop}
The estimation error
$
\eps^{- 1/(4H-2)}(\widehat H(\eps) -H)
$
is bounded in $\P_\theta$-probability, uniformly over compacts in $\Theta$, as $\eps\to 0$. 
\end{prop}
\begin{proof}
(adaptation of \cite{GH07}).
\end{proof}

\begin{rem}
In the context of Theorem \ref{thm2}, this result implies rate optimality for a particular class of loss functions of the form 
$$
\ell_M(u)= (|u|-M)^+ \wedge 1\le  \one{|u|\ge M},
$$
since the above estimator satisfies 
$$
\varlimsup_{\eps\to 0}\sup_{\theta \in K}\E_{\theta} \ell_M\Big(\eps^{- 1 /(4H-2)}(\widehat H(\eps)-H)\Big) < 1
$$
for any compact $K\subset \Theta$ and all $M$ large enough.
\end{rem} 

Estimation of $\sigma^2$ can be based on \eqref{dva} as well. The method of \cite{GH07} implies that
the estimator 
$$
\widehat \sigma^2(\eps) =  \frac{2}{c_{\widehat H(\eps)}(\psi)} \frac{\widehat Q_{J^*_\eps}}{2^{J^*_\eps(2-2\widehat H(\eps))}},
$$
where $\widehat H(\eps)$ is defined in \eqref{Hhat}, is rate optimal.

\begin{prop} 
The estimation error
$$
\eps^{- 1/(4H-2)}\frac 1 {\log \eps^{-1}}\big(\widehat \sigma^2(\eps) -\sigma^2\big)
$$
is bounded in $\P_{\theta}$-probability, uniformly over compacts in $\Theta$, as $\eps\to 0$. 
\end{prop}

Similarly, the estimators 
$$
\widetilde  \sigma^2(\eps)  = \frac  2 { c_H(\psi) } \frac{ \widehat Q_{j_\eps}(\eps) }{2^{   j_\eps (2-2H)}}
$$
with $j_\eps = \left[\frac 1 {2H-1} \log \eps^{-1}\right]$ and 
$$
\widetilde H(\eps)    =      1-  \frac 1 {2 J^*_\eps} \log_2 \left(\frac 2 {\sigma^2 c_{\widehat H(\eps)} (\psi)}  Q_{J^*_\eps}\right)
$$
are rate optimal for the corresponding parameter, when the other parameter is known.

\section{The proofs preview}\label{sec:prev}

 Let $B=(B_t, t\in \Real_+)$ and $B^H = (B^H_t, t\in \Real_+)$ be independent standard and fractional Brownian motions 
on a probability space $(\Omega, \F,\P)$. The mixed fBm \eqref{mfBm} with $\theta=(H,\sigma^2)\in (\frac 3 4, 1)\times \Real_+$ 
satisfies the canonical innovation representation \cite{Hi68} 
\begin{equation}\label{inn}
X_t = \int_0^t \rho_t(X,\theta)dt + \sqrt \eps\, \overline B_t, \quad t\in [0,T],
\end{equation}
where $\overline B$ is a Brownian motion with respect to  $\F^X_t =\sigma\{X_s, s\le t\}$, 
\begin{equation}\label{rhot}
\rho_t(X,\theta) = \int_0^t g(t,t-s;\theta)dX_s,
\end{equation}
and the function $g(t,s;\theta)$ solves the integral equation 
\begin{equation}\label{eq}
\eps g(t,s;\theta) + \int_0^t  K_{\theta}(r-s) g(t,r;\theta) dr =   K_{\theta}(s), \quad 0<s<t,
\end{equation}
with the kernel $K_\theta(\cdot)$ defined in \eqref{KH}. This equation has the unique solution in $L^2([0,t])$  
since its kernel is Hilbert-Schmidt for $H>\frac 3 4$. The stochastic integral in 
\eqref{rhot} can therefore be defined in the usual way (see \cite{PT01}).  

Let $\P^T$ and $\P^T_\theta$ be the probability measures on $C([0,T],\Real)$ induced by the Brownian motion $\sqrt \eps\, \overline B$ and the 
mixed fBm with parameter $\theta$, respectively. By the Girsanov theorem, applied to the innovation representation \eqref{inn},
$\P^T\sim \P^T_\theta$ with the Radon-Nikodym derivative   
\begin{equation}\label{RN}
\frac{d\P^T_\theta}{d\P^T} (X^T) = \exp\left(\frac  1\eps \int_0^T \rho_t(X,\theta)dX_t
-\frac 12 \frac 1 \eps \int_0^T \rho_t(X,\theta)^2dt\right).
\end{equation}
Thus asymptotics of the likelihood ratios in \eqref{LR} is determined by the limiting behavior of the solution to \eqref{eq}.

In the large time asymptotic regime with a fixed $\eps>0$ and $T\to\infty$, we show in Lemma \ref{lem:LT} that the Laplace transform 
$$
\widehat g_t(z) = \int_\Real g(t,s;\theta) e^{-zs}ds, \quad z\in \mathbb C, 
$$
where $g(t,s;\theta)$ is extended by zero outside $(0,t)$, satisfies the decomposition
$$
\widehat g_t(z)   = 1- \frac{ 1}{ X_c(-z)} + \widehat R_t(z).
$$
Here $X_c(z)$ is defined by the Cauchy type integral \eqref{Xc} of a certain explicit function.
It can be shown that the inverse Fourier transform of $\widehat g(i\lambda):= 1- 1/X(-i\lambda)$ solves 
the limit equation 
\begin{equation}\label{WHeq}
\eps g(s;\theta) + \int_0^\infty  K_{\theta}(r-s) g(r;\theta) dr =   K_{\theta}(s), \quad s\in \Real_+,
\end{equation}
obtained from \eqref{eq} by taking $t\to\infty$. This is the Wiener-Hopf equation of the second kind. 
Its solvability is not entirely obvious at the outset since the classical theory requires that 
$K_\theta \in L^1(\Real)$ (see, e.g., \cite{Krein58}). Having an essentially explicit solution 
is instrumental for deriving the formula for the Fisher information \eqref{Ieps}.

The residual $\widehat R_t(z)$ quantifies the proximity between solutions to equations 
on the finite and infinite intervals. The most challenging part of the proof is to estimate its growth 
as a function of $t$. This is what ultimately determines the correlation properties of the process 
$\rho_t(X,\theta)$ and its derivatives, see Lemmas \ref{lem1}-\ref{lem2}. To obtain suitable estimates 
we construct a representation of the solution to \eqref{eq} in terms of certain auxiliary integral equations
\eqref{pqeq}, which turn out to be more tractable for asymptotic analysis.
  
In the small noise asymptotic regime with $T$ being fixed and $\eps\to 0$, equation \eqref{eq} degenerates in the 
limit to an integral equation of the first kind, which does not have a classic solution. Nevertheless, 
the structure of kernel \eqref{KH}, corresponding to self-similarity of the fBm, implies certain 
scaling properties of the solution to \eqref{eq} (see Lemma \ref{lem:s}), which can be used to derive the small noise 
asymptotics from the large time limit.

\section{Proof of Theorem \ref{thm:main}}\label{sec:lt}
In view of the formula \eqref{RN} the likelihood ratio in Definition \ref{defn:LAN} takes the form 
\begin{equation}\label{like}
\begin{aligned} 
\log \frac{d\P^T_{\theta_0+ \phi(T) u}}{d\P^T_{\theta_0}}(X^T)  = &
\frac  1{\sqrt \eps}  \int_0^T  \big(\rho_t(X,\theta_0+ \phi(T) u)-\rho_t(X,\theta_0)\big)  d\overline B_t  \\
- & \frac 1 2\frac  1\eps  \int_0^T \big(\rho_t(X,\theta_0+ \phi(T) u)-\rho_t(X,\theta_0)\big)^2dt,
\end{aligned}
\end{equation}
where  $X$ is the mixed fBm with parameter $\theta_0$ and $\phi(T)$ is defined in Theorem \ref{thm:main}.
The matrix $I(\theta_0,\eps)$ is invertible, and establishing the LAN property in Theorem \ref{thm:main} amounts to 
proving that for any $u\in \Real^2$ 
\begin{equation}\label{qvlim}
\frac 1\eps \int_0^T \big(\rho_t(X,\theta_0+ u/\sqrt T)-\rho_t(X,\theta_0)\big)^2dt \xrightarrow[T\to\infty]{\P} u^\top I(\theta_0,\eps)u,
\end{equation}
since by the CLT for stochastic integrals \cite[Theorem 1.19]{Ku04}, this convergence also implies the convergence in distribution 
$$
\frac 1 {\sqrt \eps}\int_0^T  \big(\rho_t(X,\theta_0+u/\sqrt T)-\rho_t(X,\theta_0)\big)  d\overline B_t
\xrightarrow[T\to\infty]{d(\P)} u^\top Z
$$
with $Z\sim N\big(0, I(\theta_0,\eps)\big)$.

\subsection{Sufficient conditions}
Let us denote the partial derivatives with respect to the entries of the parameter vector $\theta$ by
$\partial_1 :=\partial_H$ and $\partial_2 :=\partial_{\sigma^2}$. The kernel in \eqref{KH} has 
partial derivatives of all orders for $\tau\ne 0$ and 
$$
\partial_i K_\theta(\cdot), \ \partial_i \partial_j K_\theta(\cdot) \in L^2([0,t]), \quad i,j\in\{1,2\}.
$$
This implies that the solution to equation \eqref{eq} also has partial derivatives 
$$
\partial_i g(t,\cdot;\theta),\ \partial_i \partial_j g(t,\cdot;\theta)\in L^2([0,t]),  \quad i,j\in\{1,2\},
$$
which can be interchanged with the stochastic integral in \eqref{rhot}. 
Consequently $\rho_t(X,\theta)$ has partial derivatives and
$$
\nabla \rho_t(X,\theta) = \int_0^t \nabla g(t,s;\theta)dX_s, \quad 
\nabla^2 \rho_t(X,\theta) = \int_0^t \nabla^2 g(t,s;\theta)dX_s,
$$
where $\nabla$ and $\nabla^2$ stand for the gradient and the Hessian operators with respect to $\theta$. 
Therefore,
\begin{equation}\label{rhodiff}
\begin{aligned}
&
 \rho_t(X,\theta_0+ u/\sqrt T)-\rho_t(X,\theta_0) = \\
&
 \frac 1 {\sqrt T}  \nabla \rho_t(X,\theta_0) u 
 +   \frac 1 {T}\int_0^1\int_0^\tau u^\top \nabla^2 \rho_t(X,\theta_0+s u/\sqrt T)  uds d\tau,
\end{aligned}
\end{equation}
and \eqref{qvlim} will be true if we manage to show that
\begin{equation}
\label{c1} 
\frac 1 \eps \frac 1 T \int_0^T   \nabla^\top \rho_t(X,\theta_0)\nabla \rho_t(X,\theta_0)dt
\xrightarrow[T\to\infty]{L^2(\Omega)}   I(\theta_0,\eps) 
\end{equation}
and, for all sufficiently small $\delta>0$,
\begin{equation}
\label{c2}
\frac 1 {T^2}\int_0^T \sup_{\theta:\|\theta-\theta_0\|\le \delta} \E \big\|\nabla^2 \rho_t(X,\theta)\big\|^2dt
\xrightarrow[T\to\infty]{}0.
\end{equation}
The main part of the proof consists of verifying properties of the gradient process $\nabla\rho_t(X,\theta)$, 
needed for establishing the latter two limits. The following lemmas provide with the key estimates to this end.

\begin{lem}\label{lem1}
Covariance function of the gradient process admits the decomposition 
$$
\E\, \nabla^\top \rho_s(X;\theta_0) \nabla \rho_t(X;\theta_0) = Q(t-s) + R(s,t)
$$
where the matrices in the right hand side satisfy the bounds
\begin{equation}\label{rbnd}
\begin{aligned}
& \big\|Q( t-s)\big\|\le    C \wedge |t-s|^{-1}\big|\log |t-s|\big|^3, \quad \forall s,t\in \Real_+,\\
&
\big\|R( s,t)\big\|\le   C   \Big(t^{-1/2} +    s^{- 1/2} + (st)^{-b}\Big),\quad \forall s,t\in [T_{\min},\infty),
\end{aligned}
\end{equation}
with some constants $b\in (0,\frac 1 2)$, $C>0$ and $T_{\min}>0$.  Moreover,  
\begin{equation}\label{diag}
Q(0)=\eps I(\theta_0,\eps).
\end{equation}
\end{lem}

\begin{lem}\label{lem2}
For all sufficiently small $\delta>0$, there exist constants $C>0$ and $T_{\min}>0$  
so that
$$
\sup_{\theta: \|\theta -\theta_0\|\le \delta}\E \big\|\nabla^2 \rho_t(X,\theta)\big\|^2\le C, \quad \forall t\ge  T_{\min}.
$$

\end{lem}

\medskip

Lemma \ref{lem1} implies that  
$
\E \nabla^\top \rho_t(X,\theta_0)\nabla \rho_t(X,\theta_0)  \xrightarrow[t\to\infty]{} \eps I(\theta_0,\eps) 
$
and 
$$
\frac 1 {T^2} \int_{0}^T \int_{0}^T  
\Big\|\E \nabla^\top \rho_s(X,\theta_0)\nabla \rho_t(X,\theta_0)\Big\|^2  dsdt\xrightarrow[T\to\infty]{} 0.
$$
In turn, these two limits ensure \eqref{c1}, since $\nabla \rho_t(X,\theta_0)$ is a centered Gaussian process.
Lemma \ref{lem2} implies the convergence in \eqref{c2}. 

\subsection{Equation \eqref{eq}}\label{sec:eq}
The probabilistic structure of the process $\rho_t(X,\theta)$ in \eqref{rhot} is determined by the solution to equation \eqref{eq}.
Thus the first step towards the proof of Lemmas \ref{lem1} and \ref{lem2} is to derive a useful representation for it.  
In this subsection we show that it can be decomposed into the sum of the principle term
independent of $t$, and the residual term which vanishes with $t\to\infty$. 
The analysis is based on the Laplace transform  
\begin{equation}\label{Lap}
\widehat g_t(z) = \int_{-\infty}^\infty g(t,s) e^{-zs}ds=\int_0^t g(t,s) e^{-zs}ds, \quad z\in \mathbb C,
\end{equation} 
where the domain of $g(t,\cdot)$ is extended to $\Real$ by zero outside the interval $(0,t)$.  
Since integration in \eqref{Lap} is carried out over a bounded interval, $\widehat g_t(z)$ is an entire function and 
hence $g(t,\cdot)$ can be recovered by the inverse Fourier transform 
$$
g(t,s) = \frac 1 {2\pi  }\int_{- \infty}^{ \infty} \widehat g_t(\i\lambda)e^{is\lambda}d\lambda.
$$

\subsubsection{The Laplace transform}
The decomposition mentioned above is based on the following representation formula for the Laplace transform. 

\begin{lem}
The Laplace transform \eqref{Lap} satisfies 
\begin{equation}\label{rep}
\widehat g_t(z) -1 = \frac{\Phi_0(z) +e^{-tz}\Phi_1(-z)}{\Lambda(z)}
\end{equation}
where 
\begin{equation}\label{Lambda}
\Lambda(z) = \eps+\frac {\sigma^2} 2  \Gamma(2H+1) \big(z^{1-2H}+(-z)^{1-2H}\big)
\end{equation}
and the functions $\Phi_0(z)$ and $\Phi_1(z)$ are sectionally 
holomorphic on $\mathbb C\setminus \Real_+$,  
\begin{equation}
\label{toinf}
\Phi_0(z) = -\eps + O(z^{1-2H})\quad \text{and}\quad \Phi_1(z) = O(z^{1-2H}),\quad z\to \infty, 
\end{equation}
and 
\begin{equation}
\label{tozero}
\Phi_0(z) =   O(z^{1-2H})\quad \text{and}\quad \Phi_1(z) = O(z^{1-2H}),\quad z\to 0.
\end{equation}
\end{lem}

\begin{proof}
By the definition of Euler's gamma function, the kernel in \eqref{KH} satisfies the integral formula 
\begin{equation}\label{Kk}
K_{\theta}(u) = \int_0^\infty \kappa(\tau) e^{-|u|\tau}d\tau, \quad u\in \Real,
\end{equation}
where 
$$
\kappa(\tau) =\sigma^2  \frac{H (2H -1)}{\Gamma(2-2H )}\tau^{1-2H}, \quad \tau\in \Real_+.
$$
Replacing the kernel in equation \eqref{eq} with expression \eqref{Kk} gives
\begin{equation}\label{geq}
\eps g(t,s) + \int_0^t g(t,r) \int_0^\infty \kappa(\tau) e^{-|s-r|\tau}d\tau  dr =   \int_0^\infty \kappa(\tau) e^{-s\tau}d\tau.
\end{equation}
The Laplace transform of the integral in the left hand side is 
\begin{align*}
&
\int_0^t \left(\int_0^t g(t,r) \int_0^\infty \kappa(\tau) e^{-|s-r|\tau}d\tau  dr\right) e^{-sz}ds =\\
&
\int_0^t g(t,r)\int_0^\infty \kappa(\tau)  \left(\int_0^t  e^{-|s-r|\tau}   e^{-sz}ds\right) d\tau dr =\\
&
\int_0^t g(t,r)\int_0^\infty \kappa(\tau)  \left(\frac{e^{-rz}-e^{-r\tau}}{\tau-z}+\frac {e^{-rz}-e^{-tz-(t-r)\tau}}{\tau+z}\right) d\tau dr =\\
&
\widehat g_t(z) \big(\mu(z) +    \mu(-z) \big)
-
\int_0^\infty  \frac{\kappa(\tau)  }{\tau-z} \widehat g_t( \tau)  d\tau  
-
e^{-tz}\int_0^\infty     \frac { \kappa(\tau)}{\tau+z} \widecheck g_t(\tau)  d\tau, 
\end{align*}
where   
$
\widecheck g_t(z):= \int_0^t g(t,t-r)e^{-z r} dr 
$
is the Laplace transform of time reversed solution and  
$$
\mu(z) =   \int_0^\infty  \frac {\kappa(x)} {x-z}dx = \frac {\sigma^2}2 \Gamma(2H+1)(-z)^{1-2H}.
$$
Similarly,  
$$
\int_0^t \left(\int_0^\infty \kappa(\tau) e^{-s\tau}d\tau\right)e^{-sz}ds =
\mu(-z)-e^{-tz}\int_0^\infty \kappa(\tau) \frac {e^{- t\tau } } {\tau+z}  d\tau.
$$ 
Thus applying the Laplace transform to \eqref{geq} we obtain  \eqref{rep} with
\begin{align*}
\Phi_0(z):= & -\eps -\mu(z)+ \int_0^\infty  \frac{\kappa(\tau)  }{\tau-z} \widehat g_t( \tau)  d\tau,
\\
\Phi_1(z) := & 
  - \int_0^\infty  \frac {\kappa(\tau)  } {\tau-z} e^{- t\tau } d\tau +\int_0^\infty     \frac { \kappa(\tau)}{\tau-z} \widecheck g_t(\tau)  d\tau,
\end{align*}
and 
$
\Lambda(z) = \eps + \mu(z)+\mu(-z).
$
The functions $\widehat g_t(\tau)$ and $\widecheck g_t(\tau)$ are bounded over $\tau\in \Real_+$ and 
the estimates \eqref{toinf}-\eqref{tozero} are derived from these formulas by standard calculations.
\end{proof}

The next lemma gathers some useful properties of $\Lambda(z)$.

\begin{lem}\label{lem:prop}
The function $\Lambda(z)$ defined in \eqref{Lambda} is non-vanishing and sectionally holomorphic on $\mathbb C\setminus \Real$
with the limits 
$$
\Lambda^\pm (\tau) = \eps +   \sigma^2a_H |\tau|^{1-2H} \begin{cases}
e^{\pm  ( H-\frac 1 2) \pi \i }, & \tau\in \Real_+, \\
e^{\mp  ( H-\frac 1 2) \pi \i }, & \tau\in \Real_-,
\end{cases}
$$
where
$
a_H = \Gamma(2H+1) \sin (\pi H).
$
These functions satisfy the symmetries 
\begin{align}
\label{sym1}
\Lambda^+(\tau) & = \overline{\Lambda^-(\tau)}, \\
\label{sym2}
\frac{\Lambda^+(\tau)}{\Lambda^-(\tau)} & =\frac{\Lambda^-(-\tau)}{\Lambda^+(-\tau)},
\end{align}
and the principal branch of the argument  $\alpha(\tau) := \arg\big\{\Lambda^+(\tau)\big\}$,
\begin{equation}\label{alpha}
\alpha(\tau) =  \arctan \frac{\sigma^2 a_H \sin\big((H-\frac 1 2)\pi  \big)}
{\eps|\tau|^{2H-1}+\sigma^2 a_H  \cos\big( (H-\frac 1 2)\pi \big)}\sign(\tau),
\end{equation}
is an odd decreasing function, continuous on $\Real\setminus \{0\}$, satisfying 
\begin{equation}\label{alas}
\alpha(0+) =   \pi  ( H-\tfrac 1 2)\quad \text{and}\quad 
\alpha(\tau)= O(\tau^{1-2H})\quad \text{as\ }\tau\to\infty.
\end{equation}
\end{lem}

\begin{proof}
All the claims are derived by direct calculations using \eqref{Lambda}.
\end{proof}

\subsubsection{An equivalent representation}
Next we will derive an alternative expression for the Laplace transform \eqref{Lap}.
The key observation to this end is that $\widehat g_t(z)$ 
is an entire function and hence discontinuity in the denominator in the right hand side 
of \eqref{rep} must be removable:
$$
\lim_{\Im(z)>0, z\to \tau}\frac{\Phi_0(z)  + e^{-tz}\Phi_1(-z)}{\Lambda(z) }
=
\lim_{\Im(z)<0, z\to \tau}\frac{\Phi_0(z)  + e^{-tz}\Phi_1(-z)}{\Lambda(z) }, \ \forall \tau\in \Real.
$$
By rearranging and using  symmetries \eqref{sym2}, this condition reduces to 
\begin{equation}\label{bndPhi}
\begin{aligned}
& 
\Phi_0^+(\tau) - \frac{\Lambda^+(\tau)}{\Lambda^-(\tau)} \Phi_0^-(\tau)  = 
e^{-t \tau } \Phi_1(-\tau) \Big(\frac{\Lambda^+(\tau)}{\Lambda^-(\tau)}-1\Big),\\
&
\Phi_1^+(\tau) - \frac{\Lambda^+(\tau)}{\Lambda^-(\tau)} \Phi_1^-(\tau)  = 
e^{-t \tau } \Phi_0(-\tau) \Big(\frac{\Lambda^+(\tau)}{\Lambda^-(\tau)}-1\Big),
\end{aligned}
\quad \forall \tau\in \Real_+,
\end{equation}
where, due to \eqref{sym1}, 
$$
\frac{\Lambda^+(t)}{\Lambda^-(t)} = \exp(2\i\alpha(t)).
$$

The functions $\Phi_0(z)$ and $\Phi_1(z)$ are sectionally holomorphic on $\mathbb C\setminus \Real_+$, 
satisfy the boundary conditions \eqref{bndPhi} and the growth estimates \eqref{toinf}. 
Using the usual technique for solving the Hilbert boundary value problems,
such functions can be expressed in terms of solutions to certain auxiliary integral equations.

The first step towards the construction of this representation consists of finding a function $X(z)$, sectionally holomorphic on 
$\mathbb C\setminus \Real_+$ and satisfying the {\em homogeneous} boundary condition, cf. \eqref{bndPhi},
$$
X^+(\tau) - \frac{\Lambda^+(\tau)}{\Lambda^-(\tau)} X^-(\tau) =0, \quad \forall \tau \in \Real_+.
$$
This is a standard instance of the Hilbert boundary value problem \cite{Gahov}. 
Since the function 
$$
\log \frac{\Lambda^+(\tau)}{\Lambda^-(\tau)} = 2\i\alpha(\tau)
$$
satisfies the H\"older condition on $\Real_+\cup \{\infty\}$, all solutions to this problem, which do not vanish on 
$\mathbb C \setminus \{0\}$, have the form $X(z) = z^k X_c(z)$ for some integer $k\in \mathbb Z$, 
where the canonical part is found by the Sokhotski--Plemelj formula
\begin{equation}\label{Xc}
\begin{aligned}
X_c(z) = &
\exp \Big(\frac 1 {2\pi \i} \int_0^\infty \frac{\log \Lambda^+(\tau)/\Lambda^-(\tau)}{\tau-z}d\tau\Big)
= \\
&
\exp \Big(\frac 1 \pi \int_0^\infty \frac{\alpha(\tau)}{\tau-z}d\tau\Big), \quad z\in \mathbb C\setminus \Real_+.
\end{aligned}
\end{equation}
The following lemma summarizes some of its useful properties. 

\begin{lem}
The function defined in \eqref{Xc} satisfies the asymptotics  
\begin{equation}\label{Xest}
X_c(z)   = \begin{cases}
O(z^{\frac 1 2-H}), & \quad z\to 0, \\
1, & \quad z\to \infty,
\end{cases}
\end{equation}
and is related to $\Lambda(z)$, defined in \eqref{Lambda}, through the identity 
\begin{equation}\label{XXL}
 X_c(z)X_c(-z) =\frac 1 {\eps} \Lambda(z), \quad z\in \mathbb C \setminus \Real.
\end{equation}

\end{lem}

\begin{proof}
The claimed asymptotics readily follows from \eqref{alas}. To prove \eqref{XXL},  we can write
\begin{align*}
& \log X_c(z)X_c(-z) = \\
& \frac 1 {2\pi \i} \int_0^\infty \frac{1}{\tau-z} \log \frac{\eps^{-1}\Lambda^+(\tau)}{\eps^{-1}\Lambda^-(\tau)}d\tau
+
\frac 1 {2\pi \i} \int_0^\infty \frac{1}{\tau+z} \log \frac{\eps^{-1}\Lambda^+(\tau)}{\eps^{-1}\Lambda^-(\tau)}d\tau.
\end{align*}
By changing the integration variable and using the symmetry \eqref{sym2}, the second integral can be written as 
$$
\frac 1 {2\pi \i} \int_0^\infty \frac{1}{\tau+z} \log \frac{\eps^{-1}\Lambda^+(\tau)}{\eps^{-1}\Lambda^-(\tau)}d\tau =
-\frac 1 {2\pi \i} \int_{-\infty}^0 \frac{1}{\tau-z} \log \frac{\eps^{-1}\Lambda^-(\tau)}{\eps^{-1}\Lambda^+(\tau)}d\tau.
$$
Since $\log \big(\Lambda^\pm(\tau)/\eps\big) = O(\tau^{1-2H})$, this implies 
$$
\log X_c(z)X_c(-z) = 
\frac 1 {2\pi \i} \int_{-\infty}^\infty \frac{\log \big( \Lambda^+(\tau)/\eps\big)}{\tau-z}  d\tau
-\frac 1 {2\pi \i} \int_{-\infty}^\infty \frac{\log \big( \Lambda^-(\tau)/\eps\big)}{\tau-z}  d\tau.
$$
The function $\Lambda(z)$ is non-vanishing and holomorphic on the lower and upper half-planes, hence each of the integrals 
can be computed by the standard contour integration. When $\Im(z)>0$ the first integral gives $\log (\Lambda(z)/\eps)$
and the second vanishes, which proves validity of \eqref{XXL} in the upper half-plane. The same argument applies to the lower half-plane.
\end{proof}

Now let us define 
\begin{equation}\label{SD}
\begin{aligned}
S(z) := & \frac{\Phi_0(z)+\Phi_1(z)}{2X(z)},
\\
D(z) := & \frac{\Phi_0(z)-\Phi_1(z)}{2X(z)}.
\end{aligned}
\end{equation}
These functions are also sectionally holomorphic on $\mathbb C\setminus \Real_+$ and, in view of \eqref{bndPhi}, satisfy the 
{\em decoupled} boundary conditions
\begin{equation}\label{SDbnd}
\begin{aligned} 
S^+(\tau) -   S^-(\tau)  = &\phantom+\ 2\i h(\tau) e^{-t \tau }S(-\tau) ,\\
D^+(\tau) -   D^-(\tau)  = &  - 2\i h(\tau)e^{-t \tau }  D(-\tau),
\end{aligned}
\qquad \forall \tau\in \Real_+,
\end{equation}
where we defined  
$$
h(\tau): =   \frac 1{2\i}\Big(\frac{X^+(\tau)}{X^-(\tau)}-1\Big)\frac {X(-\tau)}{X^+(\tau)}.
$$
This function turns out to be  real valued, 
\begin{equation}\label{h}
\begin{aligned}
h(\tau)  =  
&
\frac 1{2\i}\Big(e^{2\i\alpha(\tau)}-1\Big)\exp \Big(-\frac {2 \tau} \pi\dashint_0^\infty\frac{\alpha(s)}{s^2-\tau^2}ds-\i\alpha(\tau)\Big) =\\
&
\exp \Big(-\frac 1 \pi \int_0^\infty  \alpha'(s)\log\Big|\frac{\tau+s}{\tau-s}\Big|ds\Big)\sin\alpha(\tau),
\end{aligned}
\end{equation}
where the dashed integral is the Cauchy principle value.

In view of estimates \eqref{tozero} and \eqref{Xest}, the functions $S(-\tau)$ and $D(-\tau)$ have at most square 
integrable singularities at the origin if we choose $k\le 0$. From here on we will fix $k=0$ so that $X(z)=X_c(z)$.
This choice is not the only possible, but it makes further calculations simpler.  
Thus the expressions in the right hand side of \eqref{SDbnd} satisfy the H\"older condition on 
$\Real_+\cup\{\infty\}$ and therefore, by the Sokhotski-Plemelj theorem, the  functions  \eqref{SD} satisfy 
\begin{equation}\label{SDeq}
\begin{aligned}
S(z) & = \phantom+\frac 1 { \pi} \int_0^\infty \frac{ h(\tau) e^{-t \tau }}{\tau-z}S(-\tau)d\tau - \frac \eps 2,\\
D(z) & = -\frac 1 {\pi} \int_0^\infty \frac{h(\tau)e^{-t \tau }}{\tau-z}D(-\tau)d\tau - \frac \eps 2,
\end{aligned}\quad z\in \mathbb C\setminus \Real_+.
\end{equation}
Constants in the right hand side match the growth of $S(z)$ and $D(z)$ as $z\to\infty$ in view of 
estimates \eqref{toinf} and \eqref{Xest}.

Consider now a pair of auxiliary integral equations 
\begin{equation}\label{pqeq}
\begin{aligned}
p_t(s) =& \phantom +\ \frac 1 { \pi} \int_0^\infty \frac{ h(\tau) e^{-t \tau }}{\tau+s}p_t(\tau)d\tau - \frac 1 2,
\\
q_t(s) =& - \frac 1 { \pi} \int_0^\infty \frac{ h(\tau) e^{-t \tau }}{\tau+s}q_t(\tau)d\tau - \frac 1 2,
\end{aligned}\quad s\in \Real_+.
\end{equation}
In the next subsection we will argue that, for all sufficiently large $t$, they have unique solutions 
such that $q_t(\cdot)+\frac 1 2$ and $p_t(\cdot)+\frac1 2$ belong to $L^2(\Real_+)$. 
Setting $z:=-\tau$ for $\tau\in \Real_+$ in \eqref{SDeq} shows that  $S(-\tau)$ and $D(-\tau)$ solve \eqref{pqeq}
multiplied by $\eps$. Since by construction $S(-\tau)$ and $D(-\tau)$ are square integrable near the origin, 
due to uniqueness of the solutions to \eqref{pqeq}, they must coincide with $\eps p_t(\tau)$ and $\eps q_t(\tau)$ and, consequently,  
$$
S(z) = \eps p_t(-z)\quad \text{and}\quad
D(z)=\eps q_t(-z),
$$
where $q_t(z)$ and $p_t(z)$ are the unique sectionally holomorphic extensions to $\mathbb C \setminus \Real_-$.
Plugging these expressions along with \eqref{XXL} and \eqref{SD} into \eqref{rep} we obtain the following result.

\begin{lem}\label{lem:LT}
The Laplace transform \eqref{Lap} satisfies 
\begin{equation}\label{hatht}
\widehat g_t(z) -1 = - \frac{ 1}{ X(-z)} + \widehat R_t(z),  \quad z\in \mathbb C,
\end{equation}
where 
\begin{equation}\label{Rtz}
\widehat R_t(z):=   \frac { 1 }{X(-z)}  \big(p_t(-z)+q_t(-z)+1\big)  + e^{-tz}\frac{1}{X(z)} \big(p_t(z)-q_t(z)\big).
\end{equation}
\end{lem}

\subsubsection{The auxiliary equations}
Consider the integral operator in \eqref{pqeq}
\begin{equation}\label{Aop}
(A_t f)(s) := \frac 1 { \pi} \int_0^\infty \frac{ h(\tau) e^{-t \tau }}{\tau+s}f(\tau)d\tau.
\end{equation}
The following lemma asserts that it is a contraction on $L^2(\Real_+)$ for all sufficiently large $t$.

\begin{lem}\label{lem:contr}
For any closed ball $B\subset \Theta$, there exist $T_{\min}>0$ and $\beta\in (0,1)$ such that 
$$
\|A_t f\|\le (1-\beta) \|f\|, \quad \forall f\in L^2(\Real_+), \quad \forall t\ge T_{\min}, \ \theta \in B. 
$$
\end{lem}

\begin{proof}
The function $h(\tau)$ defined in \eqref{h} is continuous, nonnegative,  vanishes as $\tau\to \infty$ and satisfies, cf.  \eqref{alas},
$$
h(0+)=\sin \alpha(0+)= \sin \big(\pi  ( H-\tfrac 1 2)\big)  \in (0,1).
$$
Then $c := \sup_{\theta\in B}h(0+) \in (0,1)$ and there exists $r>0$ such that $h(\tau)\le \frac 1 2 c  +\frac 1 2 =: 1-\beta \in (0,1)$ 
for all $\tau \in [0,r]$. Then for any $\tau \ge 0$, 
$$
h(\tau)e^{-\tau t}\le (1-\beta) \one{\tau \le r}+\|h\|_\infty e^{-r t}\one{\tau >r} \le 1-\beta,
$$
where the last inequality  holds for all $t\ge \frac 1 r \log \frac{\|h\|_\infty}{1-\beta} =: T_{\min}$.

Thus for all $t\ge T_{\min}$ and any $f,g\in L^2(\Real_+)$, by the Cauchy--Schwarz inequality, 
\begin{align*}
&
\big|\langle g, A_t f\rangle\big| \le   \frac {1-\beta}\pi  \int_0^\infty |g(s)|   \int_0^\infty \frac{ 1}{\tau+s}|f(\tau)|d\tau ds \le \\
&
\frac {1-\beta}\pi \left(\int_0^\infty f(\tau)^2\int_0^\infty \frac{\sqrt{\tau/s}}{\tau+s}dsd\tau\right)^{1/2}
\left(\int_0^\infty g(s)^2\int_0^\infty \frac{\sqrt{s/\tau}}{\tau+s}d\tau ds\right)^{1/2} = \\
&
 (1-\beta)\|g\|\|f\|.
\end{align*} 
Hence $\|A_t f\|^2=\langle A_tf, A_tf\rangle \le (1-\beta) \|A_tf\|\|f\|$, which proves the clam.
\end{proof}

The equations in \eqref{pqeq} can be written as 
\begin{equation}\label{Af}
f + \tfrac 1 2= \pm A_t (f+\tfrac 12 ) \mp \tfrac 1 2 (A_t 1).
\end{equation}
A direct calculation shows that  $A_t1\in L^2(\Real_+)$. Hence these equations have unique solutions in $L^2(\Real_+)$  
given, e.g., by the Neumann series. The estimates for these solutions, derived in the next lemma, play the key role in the 
asymptotic analysis. 

\begin{lem}\label{lem:pqbnd}
For any closed ball $B\subset \Theta$, 
there exist  constants $r_{\max}>0$, $T_{\min}>0$ and $C>0$ such that for any $r\in [0,r_{\max}]$ and all $t\ge T_{\min}$
$$
\int_{-\infty}^\infty \big|m_t(i\lambda)\big|^2 |\lambda|^{-r}d\lambda \le  C t^{r-1}, 
$$
where $m_t(z)$ is any of the functions 
\begin{equation}\label{pqbnd}
\Big\{
p_t(z)+\tfrac 1 2,\ q_t(z)+\tfrac 1 2,\ \partial_j p_t(z), \ \partial_j q_t(z), \ \partial_i\partial_j p_t(z), \
\partial_i\partial_j q_t(z) 
\Big\}.
\end{equation}

\end{lem}

\begin{proof}
Let us start with proving the bound for the first two functions in \eqref{pqbnd}. 
Calculations are similar for both equations in \eqref{pqeq}
and we will consider the first one for definiteness.  Rearranging it as in \eqref{Af} and multiplying by $s^{-r}$ shows that the function
$\phi(s):= \big(p_t(s) +\tfrac 1 2\big)s^{-r}$ solves the equation
\begin{equation}\label{Beq}
\phi  =   B_t \phi + \psi,
\end{equation}
where $\psi(s) := -\frac 1 2(A_t 1)(s)s^{-r}$ with $A_t$ as in \eqref{Aop} and  
$$
(B_t f)(s):= \frac 1 { \pi} \int_0^\infty \frac{ h(\tau) e^{-t \tau }}{\tau+s}(\tau/s)^r f(\tau) d\tau.
$$
By applying the Minkowski inequality we get
\begin{equation}\label{calcas}
\begin{aligned}
\big\|\psi\big\| = & \left(
\int_0^\infty
\left(\frac 1 2\frac 1 { \pi} \int_0^\infty \frac{ h(\tau) e^{-t \tau }}{\tau+s}s^{-r} d\tau\right)^2
ds
\right)^{1/2} \le  \\
&
\int_0^\infty
\left(  \int_0^\infty \left(\frac{ h(\tau) e^{-t \tau }}{\tau+s}s^{-r}\right)^2 ds\right)^{1/2}
d\tau = \\
&
\int_0^\infty h(\tau) e^{-t \tau }\tau^{-\frac 1 2-r} 
\left(  \int_0^\infty  \frac{u^{-2r}}{(u+1)^2}  du\right)^{1/2}
d\tau \le C  t^{r-\frac 12} 
\end{aligned}
\end{equation}
where $r<1/2$ is assumed.  
Calculations as in the proof of Lemma \ref{lem:contr} show that $B_t$ is a contraction on $L^2(\Real_+)$ for all $t\ge T_{\min}$.
Indeed, for any $f,g\in L^2(\Real_+)$ and $r<1/4$,
\begin{align*}
&
\big|\langle g, B_tf\rangle\big| \le 
\int_0^\infty  |g(s)|\frac 1 { \pi} \int_0^\infty \frac{ h(\tau) e^{-t \tau }}{\tau+s}(\tau/s)^r |f(\tau)| d\tau ds \le \\
&
\frac {1-\beta} { \pi} \int_0^\infty \int_0^\infty |g(s)| \frac{ (s/\tau)^{\frac 1 4} }{\sqrt{\tau+s}} |f(\tau)|  \frac{ (\tau/s)^{r+\frac 1 4} }{\sqrt{\tau+s}} 
  d\tau ds \le \\
&
\frac {1-\beta}\pi 
\left( \int_0^\infty g(s)^2 \int_0^\infty  \frac{ (s/\tau)^{\frac 1 2} }{\tau+s} d\tau ds \right)^{1/2}
 \left(
   \int_0^\infty f(\tau)^2 \int_0^\infty    \frac{ (\tau/s)^{2r+\frac 1 2} }{ \tau+s } 
   dsd\tau
\right)^{1/2} =\\
&
\frac{1-\beta}{\sqrt{\cos(2\pi r)}} \|g\|\|f\|,
\end{align*}
where $\beta$ is given in Lemma \ref{lem:contr}. 
Hence $\|B_tf\|\le (1-\widetilde \beta)\|f\|$ with some $\widetilde\beta>0$ if $r$ is small enough.
This implies $\|\phi\|\le  \widetilde \beta^{-1}\|\psi\|$, that is, 
\begin{equation}\label{bb}
\left(\int_0^\infty (p_t(s)+\tfrac 1 2)^2s^{-2r}ds\right)^{1/2} \le C t^{r-\frac 1 2}.
\end{equation}
We can now prove the bound for the first function in \eqref{pqbnd},
\begin{align*}
&
\int_{-\infty}^\infty \big|p_t(\i\lambda)+\tfrac 1 2\big|^2 |\lambda|^{-r}d\lambda = 
\int_{-\infty}^\infty \left|\frac 1 { \pi} \int_0^\infty \frac{ h(\tau) e^{-t \tau }}{\tau+\i\lambda}p_t(\tau)d\tau\right|^2 |\lambda|^{-r}d\lambda \le \\
&
\int_{-\infty}^\infty \left|  \int_0^\infty \frac{ h(\tau) e^{-t \tau }}{\tau+\i\lambda}(p_t(\tau)+\tfrac 1 2)d\tau\right|^2 |\lambda|^{-r}d\lambda 
+ 
 \int_{-\infty}^\infty \left|  \int_0^\infty \frac{ h(\tau) e^{-t \tau }}{\tau+\i\lambda} d\tau\right|^2 |\lambda|^{-r}d\lambda. 
\end{align*}
Due to the Minkowski inequality, the last integral satisfies 
\begin{align*}
\int_{-\infty}^\infty \left|  \int_0^\infty \frac{ h(\tau) e^{-t \tau }}{\tau+\i\lambda} d\tau\right|^2 |\lambda|^{-r}d\lambda \le &
\left(
\int_0^\infty  h(\tau)  e^{- t \tau }\left(  \int_{-\infty}^\infty \frac{ |\lambda|^{-r}}{\tau^2+\lambda^2}   d\lambda \right)^{1/2} d\tau
\right)^2  \\
=\, &  
C 
\left(
\int_0^\infty     e^{- t \tau } \tau^{-r/2-1/2} d\tau
\right)^2 \le C t^{r-1}.
\end{align*}
The other integral can be bounded similarly,
\begin{align*}
&
\int_{-\infty}^\infty \left|  \int_0^\infty \frac{ h(\tau) e^{-t \tau }}{\tau+\i\lambda}(p_t(\tau)+\tfrac 1 2)d\tau\right|^2 |\lambda|^{-r}d\lambda \le \\
&
\left(
\int_0^\infty  h(\tau)  e^{- t \tau } \big|p_t(\tau)+\tfrac 1 2\big|
\left(  
\int_{-\infty}^\infty \frac{|\lambda|^{-r }}{\tau^2+\lambda^2}  d\lambda
\right)^{1/2}  
d\tau \right)^2 \le \\
&
C
\left(
\int_0^\infty     e^{- t \tau } \big|p_t(\tau)+\tfrac 1 2\big|
\tau^{-r/2-1/2}  
d\tau \right)^2
\le \\
&
C
\int_0^\infty     \big(p_t(\tau)+\tfrac 1 2\big)^2
\tau^{-2r   } d\tau   
\int_0^\infty       e^{-2 t \tau } \tau^{ r -1}
d\tau  \le C t^{r-1},
\end{align*}
where we used \eqref{bb} and applied the Minkowski and Cauchy-Schwarz inequalities. 
This completes the proof for the first two functions in \eqref{pqbnd}. 

The other two bounds are verified similarly. Note that $\phi(s):=\partial_j p_t(s)s^{-r}$ also solves 
the equation \eqref{Beq} with 
$$
\psi(s) := s^{-r}  \frac 1 \pi \int_0^\infty \frac{\partial_j h(\tau) e^{-t\tau}}{\tau+s}p_t(\tau)d\tau.
$$
In view of \eqref{alpha}
\begin{equation}\label{aldifas}
\partial_j \alpha(\tau) = \begin{cases}
O(1), & \tau \to 0,\\
O(\tau^{1-2H}\log \tau), & \tau \to\infty,
\end{cases}
\end{equation}
and, consequently, due to \eqref{h},
$$
\partial_j \log h(\tau) = \begin{cases}
O(1), & \tau\to 0, \\
O(\log \tau), & \tau \to \infty.
\end{cases}
$$
Calculations as in \eqref{calcas} then show that $\|\psi\|\le Ct^{r-1/2}$ and the claimed bound for the next two functions in \eqref{pqbnd}
are proved as above. The last two bounds for the second order derivatives are verified along the same lines. 
%
\end{proof}

\subsection{Proof of Lemma \ref{lem1}} 
In this subsection we will omit $\theta_0$ from the notations for brevity.
Covariance function of the gradient process satisfies 
\begin{equation}\label{specr}
\begin{aligned}
&
 \E\,\partial_i \rho_s(X) \partial_j  \rho_t(X) =  \\
& 
\int_0^t \int_0^s \partial_i g(s,s-x) \partial_j g(t,t-y) K_{\theta_0}(x-y)dxdy +  \\
& 
\eps \int_0^s \partial_i g(s,s-x) \partial_j g(t,t-x)dx =\\
& 
\int_0^t \int_0^s \partial_i g(s,x) \partial_j g(t,y) K_{\theta_0}(y-(x+t-s) )dxdy +  \\
&  
\eps \int_0^s \partial_i g(s,x) \partial_j g(t,x+t-s)dx.
\end{aligned}
\end{equation}
Restriction of $\Lambda(z)$ defined in \eqref{Lambda} to the imaginary axis is 
$$
\Lambda(\i\lambda) = \eps + \widehat K_{\theta_0}(\lambda), \quad \lambda \in \Real\setminus \{0\},
$$
where $\widehat K_{\theta_0}$ is the Fourier transform \eqref{KK}.
Hence by extending the domain of $g(t,\cdot)$ to $\Real$ by zero and 
applying Plancherel's theorem we get 
\begin{align}
\label{Q3R}
\E\,\partial_i \rho_s(X) \partial_j  \rho_t(X) = &
\frac 1{2\pi  }\int_{- \infty}^{ \infty}  \partial_j \widehat g_t(\i\lambda) \overline{\partial_i \widehat g_s(\i\lambda)}\Lambda(\i\lambda)e^{\i(t-s)\lambda}d\lambda =\\
&
\nonumber
Q_{ij}(t-s) +  R^{(1)}_{ij}(s,t)+R^{(2)}_{ij}(s,t)+R^{(3)}_{ij}(s,t),
\end{align} 
where, due to  Lemma \ref{lem:LT}, 
\begin{equation}\label{Qij}
Q_{ij}(t-s) :=  
\frac 1{2\pi  }\int_{- \infty}^{ \infty}  \partial_j \frac{ 1}{ X(-\i\lambda)}  \partial_i \frac{ 1}{ X(\i\lambda)} 
\Lambda(\i\lambda) e^{\i(t-s)\lambda}d\lambda
\end{equation}
and 
\begin{align}
\label{R1}
R^{(1)}_{ij}(s,t) := & 
-\frac 1{2\pi  }\int_{-\infty}^{\infty}  \partial_j    \frac{ 1}{ X(-\i\lambda)}  \overline{\partial_i 
\widehat R_s(\i\lambda) }\Lambda(\i\lambda)e^{\i(t-s)\lambda}d\lambda,
\\
R^{(2)}_{ij}(s,t) := & 
-\frac 1{2\pi  }\int_{-\infty}^{\infty}  \partial_j \widehat R_t(\i\lambda) \overline{\partial_i \frac{ 1}{ X(-\i\lambda)}}\Lambda(\i\lambda)e^{\i(t-s)\lambda}d\lambda,
\\
R^{(3)}_{ij}(s,t) := & \phantom +\
\frac 1{2\pi  }\int_{-\infty}^{\infty}  \partial_j \widehat R_t(\i\lambda) \overline{\partial_i \widehat R_s(\i\lambda)}\Lambda(\i\lambda)e^{\i(t-s)\lambda}d\lambda.
\end{align}
The first bound in \eqref{rbnd} is derived in the following lemma.

\begin{lem}
There exists $C>0$ such that
$$
\big|Q_{ij}(t-s)\big|\le   C  \wedge |t-s|^{-1}\big|\log |t-s|\big|^3, \quad \forall s,t\in \Real_+.
$$
\end{lem}

\begin{proof}
Let us estimate the growth of the integrand in \eqref{Qij},
$$
f(\lambda):=  \partial_j \frac{ 1}{ X(-\i\lambda)}  \partial_i \frac{ 1}{ X(\i\lambda)} 
\Lambda(\i\lambda),
$$
at the origin and at infinity. In view \eqref{aldifas}
and  \eqref{Xc},
\begin{equation}\label{logXest}
\partial_i \log X(\i\lambda) =   \frac 1 \pi \int_0^\infty \frac{\partial_\i\alpha(\tau)}{\tau-\i\lambda}d\tau
= 
\begin{cases}
O(\log |\lambda|^{-1}), & \lambda \to 0,\\
O(|\lambda|^{1-2H_0}\log |\lambda|), & \lambda \to \pm \infty.
\end{cases}
\end{equation}
Combining this estimate with \eqref{Xest} gives
\begin{equation}\label{1overX}
 \partial_i \frac{ 1}{ X(\i\lambda)}  = -\frac {\partial_i \log X(\i\lambda)} {X(\i\lambda)} 
= 
\begin{cases}
O(|\lambda|^{H_0-1/2}\log |\lambda|^{-1}), & \lambda\to 0, \\
O(|\lambda|^{1-2H_0}\log |\lambda|), & \lambda\to \pm \infty.
\end{cases}
\end{equation}
Consequently, in view of formula \eqref{Lambda},
$$
f(\lambda) = 
 \begin{cases}
O(\log^2 |\lambda|^{-1}), & \lambda \to 0, \\
O(|\lambda|^{2-4 H_0}\log^2|\lambda|), & \lambda \to\pm \infty,
\end{cases}
$$
so that $f\in L^1(\Real)$  and 
$$
|Q_{ij}(t-s)|\le  \|f\|_1.
$$
Similarly we can estimate the derivative $f'(\lambda)$ with respect to $\lambda$,
$$
f'(\lambda) = \begin{cases}
O(|\lambda|^{-1}\log^2 |\lambda|^{-1})& \lambda \to 0,\\
O(|\lambda|^{1-4 H_0}\log^2|\lambda|), & \lambda \to \pm \infty.
\end{cases}
$$
Standard bounds for the Fourier integral of such functions \cite{IL14} imply 
$$
|Q_{ij}(t-s)|\le \left|\int_{-\infty}^\infty f(\lambda) e^{\i(t-s) \lambda}d\lambda\right| \le c  |t-s|^{-1}\big|\log  |t-s|\big|^3,
$$
for some constant $c>0$. The claimed estimate follows by combining the two bounds. 
\end{proof}

The next lemma proves the second bound in \eqref{rbnd}.

\begin{lem} 
There exist constants $b\in (0,\frac 12)$, $C>0$ and $T_{\min}>0$ such that for all $s,t\ge T_{\min}$,
\begin{equation}\label{Rbnd}
\begin{aligned}
\big|R^{(1)}_{ij}(s,t)\big|& \le  C s^{ -1/2},  \\
\big|R^{(2)}_{ij}(s,t)\big|& \le  C t^{-1/2}, \\
\big|R^{(3)}_{ij}(s,t)\big|& \le  C (st)^{-b}.
\end{aligned}
\end{equation}
 
\end{lem}
 
\begin{proof}
The expression in \eqref{Rtz} satisfies the bound 
\begin{equation}\label{bbb}
\begin{aligned}
&
\big|X(\i\lambda) \partial_i \widehat R_t(\i\lambda)\big|  \le \, 
\Big| \big(p_t(-\i\lambda)+q_t(-\i\lambda)+1\big) \partial_i \log X(\i\lambda) \Big| + 
   \\
& 
  \Big| \big(p_t(\i\lambda)-q_t(\i\lambda)\big) \partial_i \log X(\i\lambda)\Big|
+2\big|\partial_ip_t(\i\lambda)\big| + 2\big| \partial_iq_t(\i\lambda)\big|,
\end{aligned}
\end{equation}
where we used the conjugacy $\overline {X(\i\lambda)}=X(-\i\lambda)$.
Hence the expression for $R^{(1)}_{ij}(s,t)$ in \eqref{R1} satisfies
\begin{align}\nonumber
\big|R^{(1)}_{ij}(s,t)\big|  \le &
 \int_{-\infty}^{\infty} \Big| \partial_j    \frac{ 1}{ X(-\i\lambda)}  \overline{\partial_i 
\widehat R_s(\i\lambda) }\Lambda(\i\lambda)\Big|d\lambda = \\
\nonumber &
  \int_{-\infty}^{\infty}      \Big|\partial_j\log  X(\i\lambda)\Big|
   \Big|X(\i\lambda)\partial_i \widehat R_s(\i\lambda) \Big|
   \left| \frac{\Lambda(\i\lambda)}{X(-\i\lambda)X(\i\lambda)}\right|d\lambda \le  \\
& \label{R1bnd}
 2 \int_{-\infty}^{\infty}      f_1(\lambda)
   \Big( \big|\partial_ip_s(\i\lambda)\big| +  \big| \partial_iq_s(\i\lambda)\big| \Big)
   d\lambda +
\\
\nonumber &   2\int_{-\infty}^{\infty} f_2(\lambda)
   \Big( \big|p_s(\i\lambda)+\tfrac 1 2\big|+\big|q_s(\i\lambda)+\tfrac 1 2\big| \Big)
   d\lambda,
\end{align}
where we used \eqref{bbb} and defined 
\begin{align*}
f_1(\lambda) :=  &
\eps \Big|\partial_j\log  X(\i\lambda) \Big|, \\
f_2(\lambda) := &  \eps \Big|\partial_j\log  X(\i\lambda)\partial_i\log  X(\i\lambda)\Big|.  
\end{align*} 
Due to the estimate \eqref{logXest}, 
$$
f_1( \lambda) = \begin{cases}
O\big(\log   |\lambda|^{-1}  \big),& \lambda\to 0,\\
O\big(|\lambda|^{1-2 H_0 }\log |\lambda|\big), &\lambda\to \pm \infty,
\end{cases}
$$
and 
$$
f_2(\lambda) = \begin{cases}
O\big( \log^2  |\lambda|^{-1}  \big),& \lambda\to 0,\\
O\big(|\lambda|^{2-4  H_0}\log^2 |\lambda|\big), &\lambda\to \pm \infty.
\end{cases}
$$
Thus  $f_1, f_2\in L^2(\Real)$. By estimate \eqref{pqbnd} with $r=0$, 
$$
\int_{-\infty}^{\infty}      f_1(\lambda)\big|\partial_ip_s(\i\lambda)\big|d\lambda 
\le \big\|f_1\big\|\,\big\|\partial_i p_s\big\| \le C s^{-1/2}.
$$
The same estimate is valid for the rest of the integrals in \eqref{R1bnd} and the first bound in \eqref{Rbnd} follows.
The second bound is proved similarly. 
To prove the third bound, note that 
\begin{align}\label{twof}
\big|R^{(3)}_{ij}(s,t)\big| \le &  
 \int_{-\infty}^{\infty} \Big| X(\i\lambda)\partial_j \widehat R_t(\i\lambda) X(\i\lambda) \partial_i \widehat R_s(\i\lambda)\Big|d\lambda\le \\
&
\nonumber
 \left(\int_{-\infty}^{\infty} \Big| X(\i\lambda)\partial_j \widehat R_t(\i\lambda)\Big|^2  d\lambda\right)^{1/2}
 \left(\int_{-\infty}^{\infty} \Big| X(\i\lambda)\partial_i \widehat R_s(\i\lambda)\Big|^2  d\lambda\right)^{1/2}.
\end{align}
In view of \eqref{bbb},
\begin{equation}\label{intbnd}
\begin{aligned}
&
\int_{-\infty}^{\infty} \Big| X(\i\lambda)\partial_j \widehat R_t(\i\lambda)\Big|^2  d\lambda  \le \\
&
4\int_{-\infty}^{\infty} \Big|
\big(p_t(-\i\lambda)+q_t(-\i\lambda)+1\big) \partial_i \log X(\i\lambda) \Big|^2   d\lambda
+ \\
&
4\int_{-\infty}^{\infty} \Big|
\big(p_t(\i\lambda)-q_t(\i\lambda)\big) \partial_i \log X(\i\lambda) \Big|^2   d\lambda
+ \\ 
&
8\int_{-\infty}^{\infty} \Big|\partial_ip_t(\i\lambda)\Big|^2   d\lambda
+
8\int_{-\infty}^{\infty} \Big|\partial_iq_t(\i\lambda)\Big|^2   d\lambda.
 \end{aligned}
\end{equation} 
Due to \eqref{logXest},
$
\big|\partial_i \log X(\i\lambda,\eta) \big|^2   \le C |\lambda|^{-r}
$
for any $r\in (0,1)$. Hence, with $r>0$ small enough, Lemma \ref{lem:pqbnd} guarantees that 
all the integrals in \eqref{intbnd} are bounded by $C t^{r-1}$. Applying the same argument to the second term in 
\eqref{twof} we conclude that 
$$
\big|R^{(3)}_{ij}(s,t)\big|\le C s^{r/2-1/2}t^{r/2-1/2}.
$$ 
This verifies the last bound in \eqref{Rbnd} with $b:=1/2-r/2\in (0,1/2)$.
\end{proof}
 
Finally, the next lemma verifies formula \eqref{diag}. 

\begin{lem} 
\begin{align*}
Q_{ij}(0)=
 \frac \eps {4\pi} \int_{-\infty}^\infty \partial_i \log\big(\eps+\widehat K_{\theta_0}(\lambda)\big)
\partial_j \log\big(\eps+\widehat K_{\theta_0}(\lambda)\big) d\lambda.
\end{align*}
\end{lem}

\begin{proof}
In view of \eqref{XXL} the expression in \eqref{Qij} can be written as
\begin{align*}
Q_{ij}(0) = &
\frac 1{2\pi  }\int_{- \infty}^{ \infty}  \partial_j \frac{ 1}{ X(-\i\lambda)}  \partial_i \frac{ 1}{ X(\i\lambda)} 
\Lambda(\i\lambda)  d\lambda =\\
&
\frac \eps {2\pi  }\int_{- \infty}^{ \infty}  \partial_i\log X(\i\lambda) \partial_j \log X(-\i\lambda)  
   d\lambda.
\end{align*}
On the other hand,  
\begin{align*}
&
\frac \eps {4\pi} \int_{-\infty}^\infty \partial_i\log \Lambda(\i\lambda) \partial_j \log \Lambda(\i\lambda)
 d\lambda  =\\
&
\frac \eps {2\pi} \int_{-\infty}^\infty \partial_i\log X(\i\lambda) \partial_j \log X(-\i\lambda)
 d\lambda
+
\frac \eps {2\pi} \int_{-\infty}^\infty \partial_i\log X( \i\lambda) \partial_j \log X( \i\lambda)
 d\lambda.
\end{align*}
Hence the formula in question  is true if we show that the latter integral vanishes. 
In view of \eqref{Xc},
\begin{align*}
&
\int_{-\infty}^\infty \partial_i\log X( \i\lambda) \partial_j \log X( \i\lambda) 
 d\lambda =\\
&
 \frac 1 {\pi^2} \int_0^\infty \int_0^\infty \partial_\i\alpha(\tau)\partial_j\alpha(r)   
 \left(\int_{-\infty}^\infty \frac{1}{\tau-\i\lambda} \frac{1}{r-\i\lambda}   d\lambda  \right)d\tau dr =0.
\end{align*}
The last equality holds since for any $r, \tau \in \Real_+$ the integral in the brackets vanishes,
as can be readily checked by the standard contour integration. 
\end{proof}

\subsection{Proof of Lemma \ref{lem2}}
This lemma involves only one dimensional distributions of the process $\rho_t(X,\theta)$ and its partial derivatives. 
On the other hand, unlike in Lemma \ref{lem1}, $\theta$ may be distinct from $\theta_0$, the true value of the parameter, 
which determines the distribution of the sample $X^T$. In this subsection, we will stress this distinction by adding 
the relevant parameter value to the notations.

We have to show that for all sufficiently small $\delta>0$ there exist constants $C>0$ and  $T_{\min}>0$ 
such that 
$$
\sup_{\|\theta-\theta_0\|\le \delta}\E \big(\partial_i\partial_j \rho_t(X,\theta)\big)^2 \le C, \quad \forall t\ge T_{\min}.
$$
Similarly to \eqref{specr}-\eqref{Q3R}
\begin{align*}
&
\E \big(\partial_i\partial_j \rho_t(X,\theta)\big)^2 
= 
\frac 1{2\pi  }\int_{- \infty}^{ \infty}  \Big|\partial_i\partial_j \widehat g_t(\i\lambda;\theta)\Big|^2 \Lambda(\i\lambda;\theta_0) d\lambda  
\le \\
&
\int_{- \infty}^{ \infty}  \Big|\partial_i\partial_j \frac 1 {X(\i\lambda;\theta)}\Big|^2 \Lambda(\i\lambda;\theta_0) d\lambda  
+
\int_{- \infty}^{ \infty}  \Big|\partial_i\partial_j \widehat R_t(\i\lambda;\theta)\Big|^2 \Lambda(\i\lambda;\theta_0) d\lambda,
\end{align*}
where the bound holds due to decomposition \eqref{hatht}. It remains to prove that both terms in the right hand side are bounded
functions of $t\in [T_{\min},\infty)$ for some $T_{\min}>0$, uniformly over $\theta$ in a $\delta$-vicinity of $\theta_0$.
This is done in the following two lemmas. 

\begin{lem} 
For all sufficiently small $\delta>0$, there exists a constant $C>0$ such that 
$$
\sup_{\|\theta-\theta_0\|\le \delta}\int_{- \infty}^{ \infty}  \Big|\partial_i\partial_j \frac 1 {X(\i\lambda;\theta)}\Big|^2 \Lambda(\i\lambda;\theta_0) d\lambda \le C.
$$
\end{lem}

\begin{proof}
The second order derivatives  of $\alpha(\tau,\theta)$ defined in \eqref{alpha} are continuous in $\tau$ and satisfy, cf. \eqref{aldifas}
$$
\partial_i \partial_j \alpha(\tau,\theta)= \begin{cases}
O(1), & \tau \to 0, \\
O(\tau^{1-2H}\log^2\tau), & \tau \to \infty.  
\end{cases}
$$
Consequently 
$$
\partial_i \partial_j \log X(\i\lambda,\theta)=
  \frac 1 \pi \int_0^\infty \frac{\partial_i\partial_j \alpha(\tau,\theta)}{\tau-\i\lambda}d\tau
= 
\begin{cases}
O(\log  |\lambda|^{-1}), & \lambda \to 0,\\
O(|\lambda|^{1-2H}\log^2 |\lambda|), & \lambda \to \pm \infty,
\end{cases}
$$
and in view of \eqref{Xest} and \eqref{logXest},
\begin{align*} 
\partial_i\partial_j \frac 1 {X(\i\lambda;\theta)} =\, &
\frac 1{X(\i\lambda;\theta)}
\Big(
\partial_i \log X(\i\lambda;\theta) \partial_j \log X(\i\lambda;\theta)-\partial_i\partial_j \log X(\i\lambda;\theta) 
\Big) =\\
&
\begin{cases}
|\lambda|^{H-1/2}\log^2|\lambda|^{-1}, &\lambda \to 0, \\
|\lambda|^{1-2H}\log^2|\lambda|,  & \lambda\to \infty.
\end{cases}
\end{align*}
This estimate and \eqref{Lambda} imply   
\begin{equation}\label{ppX}
\Big|\partial_i\partial_j \frac 1 {X(\i\lambda;\theta)}\Big|^2 \Lambda(\i\lambda;\theta_0) = 
\begin{cases}
|\lambda|^{-2\delta}\log^4|\lambda|^{-1}, &\lambda \to 0, \\
|\lambda|^{2-4H}\log^4|\lambda|,  & \lambda\to \infty.
\end{cases}
\end{equation}
This function is integrable on $\Real$ for all sufficiently small $\delta>0$ which verifies the claim.
\end{proof}

\begin{lem}
For all sufficiently small $\delta>0$, there exist positive constants $C$, $T_{\min}$ and $c$ such that 
$$
\sup_{\|\theta-\theta_0\|\le \delta}\int_{- \infty}^{ \infty}  \Big|\partial_i\partial_j \widehat R_t(\i\lambda;\theta)\Big|^2 \Lambda(\i\lambda;\theta_0) d\lambda \le C t^{-c}, 
\quad \forall t\ge T_{\min}.
$$
\end{lem}

\begin{proof}
In view of formula \eqref{Rtz}, it suffices to show that for all sufficiently small $\delta>0$, there exist positive constants 
$C$, $T_{\min}$ and $c$ such that 
\begin{equation}\label{trib}
\begin{aligned}
&
I_1(t):= \int_{- \infty}^{ \infty}  \Big|\partial_i\partial_j \frac{1}{X(\i\lambda,\theta)}\big(p_t(\i\lambda,\theta)+\tfrac 1 2 \big)\Big|^2 \Lambda(\i\lambda;\theta_0) d\lambda \le C t^{-c},
\\
&
I_2(t) := \int_{- \infty}^{ \infty}  \Big|\partial_j \frac{1}{X(\i\lambda,\theta)}\partial_i p_t(\i\lambda,\theta) \Big|^2 \Lambda(\i\lambda;\theta_0) d\lambda \le C t^{-c},
\\
&
I_3(t) := \int_{- \infty}^{ \infty}  \Big|\frac{1}{X(\i\lambda,\theta)}\partial_i\partial_j  p_t(\i\lambda,\theta) \Big|^2 \Lambda(\i\lambda;\theta_0) d\lambda \le C t^{-c},
\end{aligned}
\end{equation}
for all $\theta$ such that $\|\theta-\theta_0\|\le\delta$ and all $t\ge T_{\min}$. 
The same bounds are obviously true for $q_t(\i\lambda, \theta)$ and its derivatives as well. 

Take an $r>0$ small enough so that the assertion of Lemma \ref{lem:pqbnd} holds. 
Then for any sufficiently small $\delta>0$, the estimate \eqref{ppX} implies that 
$$
\Big|\partial_i\partial_j \frac 1 {X(\i\lambda;\theta)}\Big|^2 \Lambda(\i\lambda;\theta_0) \le C_1 |\lambda|^{-r}
$$
for some constant $C_1>0$, and the first bound in \eqref{trib} holds with $c=1-r$ by Lemma \ref{lem:pqbnd}.
The second bound holds by the same argument since, in view of \eqref{1overX} and \eqref{Lambda},
$$
\Big|\partial_j \frac{1}{X(\i\lambda,\theta)}  \Big|^2 \Lambda(\i\lambda;\theta_0)
=
\begin{cases}
O(|\lambda|^{-2\delta}\log^2 |\lambda|^{-1}), & \lambda\to 0, \\
O(|\lambda|^{2-4H}\log^2 |\lambda|), & \lambda\to \pm \infty.
\end{cases}
$$
To prove the third bound, note that by \eqref{XXL} and \eqref{Lambda},
$$
\Big|\frac{1}{X(\i\lambda,\theta)}  \Big|^2 \Lambda(\i\lambda;\theta_0)-\eps = 
\eps \Big(\frac{\Lambda(\i\lambda;\theta_0)}{\Lambda(\i\lambda;\theta)}-1\Big)=
\begin{cases}
O(|\lambda|^{-2\delta}), & \lambda\to 0, \\
O(|\lambda|^{1-2H_0+2\delta}), & \lambda\to \pm \infty.
\end{cases}
$$
Thus 
\begin{align*}
I_3(t) \le\, & \eps \int_{- \infty}^{ \infty} \Big|\partial_i\partial_j  p_t(\i\lambda,\theta) \Big|^2 
d\lambda  + \\
&
\int_{- \infty}^{ \infty} \Big|\partial_i\partial_j  p_t(\i\lambda,\theta) \Big|^2 
\Big|\Big|\frac{1}{X(\i\lambda,\theta)}\Big|^2\Lambda(\i\lambda;\theta_0)-\eps\Big|d\lambda\le C t^{-1} + C t^{1-r},
\end{align*}
where the last bound is true due to Lemma \ref{lem:pqbnd}.
\end{proof}

\section{Proofs of Theorems \ref{thm2} and \ref{thm3}}\label{sec:sn}

The LAN property in the small noise setting is derived from the large time asymptotics.  
It will be convenient to change some notations in order to to emphasize the more relevant variables. 
In particular, we will indicate the dependence of solution to \eqref{eq} on $\eps$ by the subscript and keep in mind its 
dependence on $\theta$, omitting it from the notations. Thus the equation \eqref{eq} reads 
\begin{equation}\label{eqeq}
\eps g_\eps(t,s) + \int_0^t \sigma^2 c_H |s-r|^{2H-2} g_\eps(t,r) dr =  \sigma^2 c_H s^{2H-2} , \quad 0<s<t,
\end{equation}
where we defined $c_H=H(2H-1)$.

\subsection{The key lemmas}
The following lemma reveals a useful relation between derivatives of $g_\eps(t,s)$ with respect to the parameter  and time variables. 

\begin{lem}\label{lem:inst}
The solution to \eqref{eqeq} with $\eps=1$ satisfies    
$$
t \frac{\partial}{\partial t}g_1(t,s) + s\frac{\partial }{\partial s} g_1(t,s) + g_1(t,s) =   (2H-1)\sigma^2
\frac{\partial }{\partial \sigma^2 }g_1(t,s), \quad 0<s<t.
$$
\end{lem} 
\begin{proof}
The function $g_1(t,s)$ diverges to $\infty$ as $s\to 0$, which makes a useful differentiation formula from \cite{VP80} 
inapplicable, cf. \eqref{Veq} below. 
To avoid this difficulty, define the function $h(s,t) = s g_1(t,s)$, then 
$$
s\frac{\partial }{\partial s} g_1(t,s) + g_1(t,s) = \frac {\partial}{\partial s} \big(s g_1(t,s)\big).
$$
Multiplying the equation 
\begin{equation}\label{g1}
g_1(t,s) + \int_0^t \sigma^2 c_H |s-r|^{2H-2} g_1(t,r) dr =  \sigma^2 c_H s^{2H-2}, \quad 0<s<t,
\end{equation}
by $s$ and rearranging terms gives
\begin{multline}\label{hst}
h(s,t) + \int_0^t \sigma^2 c_H |s-r|^{2H-2} h(r,t) dr 
= \\
\sigma^2 c_H\left(\int_0^t  |s-r|^{2H-2}(r-s) g_1(t,r) dr
+   s^{2H-1} \right).
\end{multline}
The expression in the brackets in the right hand side is differentiable in $s$ with the derivative 
\begin{align*}
&
\frac {\partial}{\partial s}
\left(s^{2H-1}+\int_0^t  |s-r|^{2H-2}(r-s) g_1(t,r) dr\right) =\\
&
\frac {\partial}{\partial s}
\left(
s^{2H-1}
-\int_0^s  (s-r)^{2H-1}  g_1(t,r) dr + \int_s^t  (r-s)^{2H-1}  g_1(t,r) dr
\right) =\\
&
(2H-1)\left( 
s^{2H-2}
-\int_0^s  (s-r)^{2H-2}  g_1(t,r) dr -  \int_s^t  (r-s)^{2H-2}  g_1(t,r) dr
\right) =\\
&
(2H-1)\left( 
s^{2H-2}-\int_0^s  |s-r|^{2H-2}  g_1(t,r) dr     
\right) =
 \frac{2H-1}{c_H\sigma^2}g_1(t,s).
\end{align*}
Since the solution $h(s,t)$ is differentiable at any $s\in (0,t)$ (see \cite{VP80})
\begin{equation}\label{Veq}
\begin{aligned}
& 
\frac \partial{\partial s} \int_0^t h(r,t) |s-r|^{2H-2} dr =  \\
&
\int_0^t |s-r|^{2H-2}  \frac \partial{\partial r}h(r,t) dr+h(0,t)s^{2H-2}
-h(t,t)(t-s)^{2H-2},
\end{aligned}
\end{equation}
where $h(0,t) =0$ by \eqref{hst}. Thus the right hand side of \eqref{hst} is differentiable and in view of the above formulas 
\begin{multline}\label{addme}
\frac \partial{\partial s}  h(s,t) + \int_0^t \sigma^2 c_H|s-r|^{2H-2}  \frac \partial{\partial r}h(r,t) dr 
=  \\
 (2H-1) g_1(t,s) + \sigma^2 c_H t g_1(t,t)(t-s)^{2H-2}.
\end{multline}
Arguing differentiability of $g_1(t,s)$ with respect to $t$ as in \cite[Lemma 3.5(i)]{CCK} and taking the derivative of \eqref{g1} we get 
$$
\frac{\partial} {\partial t} g_1(t,s) + \int_0^t \sigma^2 c_H |s-r|^{2H-2} \frac{\partial} {\partial t} g_1(t,r) dr 
=-   \sigma^2 c_H g_1(t,t) (t-s)^{2H-2} .
$$
Multiplying this equation by $t$ and adding the result to \eqref{addme} gives
\begin{equation}\label{thateq}
\begin{aligned} 
&
\Big( t\frac{\partial} {\partial t} g_1(t,s) + \frac \partial{\partial s}  h(s,t)\Big) + \\
&
\int_0^t \sigma^2 c_H |s-r|^{2H-2} \Big(
t \frac{\partial} {\partial t} g_1(t,r) +  \frac \partial{\partial r}h(r,t)\Big) dr 
=  
 (2H-1) g_1(t,s).
\end{aligned}
\end{equation}
On the other hand, differentiating \eqref{g1} with respect to $\sigma^2$ shows that 
\begin{multline*}
\frac{\partial }{\partial \sigma^2}g_1(t,s) +  \int_0^t \sigma^2 c_H |s-r|^{2H-2} \frac{\partial }{\partial \sigma^2} g_1(t,r) dr
+ \\ \int_0^t   c_H |s-r|^{2H-2} g_1(t,r) dr =    c_H s^{2H-2},
\end{multline*}
or equivalently, 
$$
\frac{\partial }{\partial \sigma^2}g_1(t,s) +  \int_0^t \sigma^2 c_H |s-r|^{2H-2} \frac{\partial }{\partial \sigma^2} g_1(t,r) dr
= \frac 1 {\sigma^2} g_1(t,s).
$$
Comparing this equation to \eqref{thateq} we conclude that 
$$
 t\frac{\partial} {\partial t} g_1(t,s) + \frac \partial{\partial s}  h(s,t) = (2H-1) \sigma^2 \frac{\partial }{\partial \sigma^2}g_1(t,s)
$$
by uniqnuness of the solution. 
\end{proof}

The solution to \eqref{eqeq} satisfies the following pivotal scaling property with respect to $\eps$. 

\begin{lem}\label{lem:s}
Let $\gamma = 1/(2H-1)$ and define
$$
M(\eps, \theta) = \begin{pmatrix}
1 & -2\sigma^2  \log\eps^{-\gamma}\\
0 & 1
\end{pmatrix}, \quad 
\nu(\eps,\theta) =
\begin{pmatrix}
4\sigma^2 \log^2\eps^{-\gamma}  & -2\log \eps^{-\gamma} \\
-2\log \eps^{-\gamma} & 0 
\end{pmatrix}.
$$
Then for any $\eps>0$ and $t > s > 0$,
\begin{align}
\label{s1}
g_\eps(t,s) =\, & \eps^{-\gamma} g_1 \big(t\eps^{-\gamma}, s\eps^{-\gamma}\big), \\
\label{s2}
\nabla g_\eps(t,s)    =\, &  \eps^{-\gamma}  \nabla g_1 \big(t\eps^{-\gamma}, s\eps^{-\gamma}\big)M(\eps,\theta)^\top, \\
\nonumber 
\nabla^2 g_\eps(t,s) =\, & 
\eps^{-\gamma} \nu(\eps,\theta)\frac{\partial}{\partial \sigma^2 } g_1 \big(t\eps^{-\gamma}, s\eps^{-\gamma}\big)
\\
& \label{s3}
+ \eps^{-\gamma} M(\eps, \theta) \nabla^2 g_1 \big(t\eps^{-\gamma}, s\eps^{-\gamma}\big) M(\eps, \theta)^\top.
\end{align}

\end{lem}

\begin{proof}
Identity \eqref{s1} is obtained by scaling all the variables in equation \eqref{g1}
by $\eps^{-\gamma}$. To verify the identities for derivatives it will be convenient to use the short notations 
\begin{align*}
& g_1'(t,s):=\frac{\partial}{\partial H} g_1(t,s), \\
& g_1^{\bullet}(t,s):=\frac{\partial}{\partial \sigma^2} g_1(t,s),  \\
& g_1''(t,s) : =\frac{\partial^2}{\partial H^2}  g_1(t,s), \\
& g_1^{\prime \bullet }(t,s) : =\frac{\partial}{\partial H} \frac{\partial}{\partial \sigma^2} g_1(t,s), \\
& g_1^{\bullet \bullet }(t,s) : = \frac{\partial^2}{{\partial \sigma^2}^2} g_1(t,s),
\end{align*}
and define the variables $u:= s\eps^{-\gamma}$ and $v:=t \eps^{-\gamma}$.
Then 
$$
\frac \partial {\partial \sigma^2}g_\eps(t,s) =   \eps^{-\gamma}   g_1^\bullet \big(v, u\big)
$$  
and, in view of Lemma \ref{lem:inst},
\begin{align*}
&
\frac \partial {\partial H}g_\eps(t,s) = 
\eps^{-\gamma}      g_1'  (v,u ) + \\
&
\frac  {\partial \gamma } {\partial H} \Big(
 \eps^{-\gamma} \log\eps^{-1}  g_1  (v,u)
+
\eps^{-\gamma}  \frac{\partial u}{\partial \gamma}\frac{\partial}{\partial u}   g_1  (v,u )  
+
\eps^{-\gamma}  \frac{\partial v}{\partial \gamma}\frac{\partial}{\partial  v}   g_1  (v,u) 
\Big) =\\
&
\eps^{-\gamma}     g_1'  (v,u)  
-2\gamma^2 \eps^{-\gamma} \log\eps^{-1} \Big(
g_1  (v,u) + u\frac{\partial}{\partial u}   g_1  (v,u)  + v\frac{\partial}{\partial  v}   g_1  (v,u) 
\Big) =\\
&
\eps^{-\gamma} \Big(  g_1'  (v,u ) - 2 \log\eps^{-\gamma}   \sigma^2 g_1^{\bullet}(v,u)\Big),
\end{align*}
which verifies \eqref{s2}. Taking another derivative with respect to $H$ we get
\begin{align*}
\frac{\partial^2}{\partial H^2} g_\eps(s,t)= &
-2\gamma^2 \eps^{-\gamma} \log\eps^{-1}
  \Big(  g_1'  (v,u) - 2 \log\eps^{-\gamma}   \sigma^2 g_1^{\bullet}(v,u)\Big) 
+ \\
&
\eps^{-\gamma} 
\frac{\partial}{\partial H} \Big(  g_1'  (v,u) - 2 \log\eps^{-\gamma}   \sigma^2 g_1^{\bullet}(v,u)\Big).
\end{align*}
Here 
\begin{align*}
\frac{\partial}{\partial H}  g_1'  (v,u) = & 
g_1''(v,u) -2\gamma^2 \log \eps^{-1} \Big(u\frac{\partial}{\partial u}g_1'(v,u) + v\frac{\partial}{\partial v}g_1'(v,u) 
\Big).
\end{align*}
By Lemma \ref{lem:inst}
$$
v \frac{\partial}{\partial v}g_1'(v,u) + u\frac{\partial }{\partial u} g_1'(v,u) + g_1'(v,u) =    2\sigma^2 
 g_1^{\bullet}(v,u) 
 +
  \frac{\sigma^2 }{\gamma}
 g_1^{\bullet\prime}(v,u)
$$
and hence 
$$
\frac{\partial}{\partial H}  g_1'  (v,u) =   
g_1''(v,u) -2\gamma^2 \log \eps^{-1} \Big(
2\sigma^2 
 g_1^{\bullet}(v,u) 
 +
  \frac{\sigma^2 }{\gamma}
 g_1^{\bullet\prime}(v,u)-g_1'(v,u)
\Big).
$$
Similarly,  
\begin{align*}
&
\frac{\partial}{\partial H}  \Big( \gamma     g_1^{\bullet}(v,u) \Big) = 
 -2\gamma^2 g_1^{\bullet}(v,u) + \gamma  \frac{\partial}{\partial H}   g_1^{\bullet}(v,u) =\\
&
-2\gamma^2 g_1^{\bullet}(v,u) + \gamma   \Big(  g_1^{\bullet'}(v,u) -2\gamma^2 \log\eps^{-1}\Big(u\frac{\partial}{\partial u}g_1^\bullet(v,u) 
+v\frac{\partial}{\partial v}g_1^\bullet(v,u) 
  \Big) \Big).
\end{align*}
By Lemma \ref{lem:inst}
$$
v \frac{\partial}{\partial v}g_1^\bullet (v,u) + u\frac{\partial }{\partial u} g_1^\bullet (v,u) + g_1^\bullet (v,u) =    
  \frac{1}{\gamma} g_1^{\bullet}(v,u)
  +
  \frac{\sigma^2 }{\gamma} g_1^{\bullet\bullet}(v,u),
$$
and hence 
\begin{align*}
\frac{\partial}{\partial H}  \Big( \gamma     g_1^{\bullet}(v,u) \Big) = &
-2\gamma^2 g_1^{\bullet}(v,u) + \gamma     g_1^{\bullet'}(v,u) \\
&
-2\gamma^2 \log\eps^{-\gamma}\Big(\frac{1}{\gamma} g_1^{\bullet}(v,u)
  +
  \frac{\sigma^2 }{\gamma} g_1^{\bullet\bullet}(u,v)- g_1^\bullet (v,u)
  \Big)  .
\end{align*}
Plugging these equations we get 
\begin{align*}
\frac{\partial^2}{\partial H^2} g_\eps(t,s)=\, & \eps^{-\gamma} 
\Big(
 g_1''(v,u)    
  -4    \log \eps^{-\gamma} \sigma^2   g_1^{\bullet\prime}(v,u)
  +
4    \sigma^4 \log^2\eps^{-\gamma}        g_1^{\bullet\bullet}(v,u)  
\Big) \\
&
+4       \eps^{-\gamma}  \log^2\eps^{-\gamma}  \sigma^2   g_1^{\bullet}(v,u).
\end{align*}
The other two second order derivatives are  
\begin{align*}
\frac{\partial^2 }{(\partial  \sigma^2)^2} g_\eps (t,s) = &
  \eps^{-\gamma} g_1^{\bullet\bullet} \big(t\eps^{-\gamma}, s\eps^{-\gamma}\big), \\
\frac{\partial^2}{\partial \sigma^2 \partial H} g_\eps (t,s) = & 
\eps^{-\gamma} \Big(  g_1^{\prime\bullet}  (v,u) 
- 2 \log\eps^{-\gamma}   \sigma^2 g_1^{\bullet\bullet}(v,u)- 2 \log\eps^{-\gamma}    g_1^{\bullet}(v,u)
\Big).
\end{align*}
In matrix notation this gives \eqref{s3}.
\end{proof}

\subsection{Proof of Theorem \ref{thm2}}

The relevant likelihood ratio is, cf. \eqref{like},
\begin{equation}\label{LReps}
\begin{aligned} 
\log \frac{d\P^\eps_{\theta_0+ \phi(\eps) u}}{d\P^\eps_{\theta_0}}(X^T)  = &
\frac  1{\sqrt \eps}  \int_0^T  \big(\rho^\eps_t(X^\eps,\theta_0+\phi(\eps)u)-\rho^\eps_t(X^\eps,\theta_0)\big)  d\overline B_t  \\
- & \frac 1 2\frac  1\eps  \int_0^T \big(\rho^\eps_t(X^\eps,\theta_0+ \phi(\eps)u)-\rho^\eps_t(X^\eps,\theta_0)\big)^2dt,
\end{aligned}
\end{equation}
where $T$ is fixed and dependence on $\eps$ is emphasized by superscripts. Here, cf. \eqref{rhodiff}, 
$$
\begin{aligned}
&
 \rho^\eps_t(X^\eps,\theta_0+ \phi(\eps) u)-\rho^\eps_t(X^\eps,\theta_0) = \\
&
    \nabla \rho^\eps_t(X^\eps,\theta_0) \phi(\eps) u 
 +    \int_0^1\int_0^\tau u^\top \phi(\eps)^\top  \nabla^2 \rho^\eps_t(X^\eps,\theta_0+s \phi(\eps) u ) \phi(\eps) uds d\tau.
\end{aligned}
$$
We will argue that for an appropriate choice of $\phi(\eps)=\phi(\eps,\theta_0)$ 
\begin{equation}\label{une}
\frac 1 {\eps }  u^\top \phi(\eps)^\top \left(\int_0^T  \nabla^\top \rho^\eps_t(X^\eps,\theta_0) \nabla  \rho^\eps_t(X^\eps,\theta_0) dt\right) \phi(\eps) u
\xrightarrow[\eps\to 0] {\P_{\theta_0} } \|u\|^2
\end{equation}
and
\begin{equation}\label{doux}
\E \frac  1\eps  \int_0^T\left(\int_0^1\int_0^\tau u^\top \phi(\eps)^\top  \nabla^2 \rho^\eps_t\big(X^\eps,\theta_0+s \phi(\eps) u \big) \phi(\eps) uds d\tau\right)^2
dt \xrightarrow[\eps\to 0]{ } 0.
\end{equation}
Then the second term in \eqref{LReps} converges to $-\frac 1 2 \|u\|^2$  in probability and the stochastic integral 
converges in distribution to $u^\top  Z$ with $Z\sim N(0,\mathrm{Id})$, see  \cite[Ch. IX.5]{JS}.

In view of Lemma \ref{lem:s},  
\begin{align*}
&
\nabla \rho^\eps_t(X^\eps,\theta_0) =   \int_0^t \nabla g_\eps(t,t-s)dX^\eps_s = \\
&
\int_0^t \nabla g_\eps(t,t-s)\sigma_0 dB^{H_0}_s + \sqrt \eps\int_0^t \nabla g_\eps(t,t-s)dB_s =\\
&
\int_0^t 
\eps^{-\gamma_0}  \nabla g_1 \big(t\eps^{-\gamma_0}, (t-s)  \eps^{-\gamma_0}\big)M(\eps,\theta_0)^\top  
\sigma_0 dB^{H_0}_s + \\
&
\sqrt \eps\int_0^t  
\eps^{-\gamma_0}  \nabla g_1 \big(t\eps^{-\gamma_0}, (t-s)  \eps^{-\gamma_0}\big)M(\eps,\theta_0)^\top 
dB_s \stackrel{d}=\\
&
\eps^{-\gamma_0}\eps^{\gamma_0H_0}\left(\int_0^{t\eps^{-\gamma_0}} 
  \nabla g_1 \big(t\eps^{-\gamma_0},  t\eps^{-\gamma_0} -s   \big) 
\sigma_0 dB^{H_0}_s\right)M(\eps,\theta_0)^\top  + \\
&
\eps^{1/2-\gamma_0} \eps^{\gamma_0/2}\left(\int_0^{t\eps^{-\gamma_0}}  
  \nabla g_1 \big(t\eps^{-\gamma_0}, t \eps^{-\gamma_0}-s  \big)
dB_s \right) M(\eps,\theta_0)^\top  =\\
&
\eps^{(1-\gamma_0)/2} \nabla \rho^1_{t\eps^{-\gamma_0}}(X^1,\theta_0)M(\eps,\theta_0)^\top,
\end{align*}
where the equality in distribution holds by the self-similarity of the fBm. 
This equality holds simultaneously for all $t\in \Real_+$
and hence the two processes coincide in distribution. 
Consequently 
\begin{align*}
&
\frac 1 {\eps }   \phi(\eps)^\top \int_0^T  \nabla^\top \rho^\eps_t(X^\eps,\theta_0) \nabla  \rho^\eps_t(X^\eps,\theta_0) dt 
\phi(\eps) \stackrel{d}= \\
&
 \eps^{ -\gamma_0}  \phi(\eps)^\top M(\eps,\theta_0)
 \bigg(\int_0^T  \nabla^\top \rho^1_{t\eps^{-\gamma_0}}(X^1,\theta_0) \nabla \rho^1_{t\eps^{-\gamma_0}}(X^1,\theta_0) dt\bigg) 
 M(\eps,\theta_0)^\top\phi(\eps) =\\
&
 T\eps^{-\gamma_0}\phi(\eps)^\top M(\eps,\theta_0)
 \bigg(\frac{1}{T\eps^{-\gamma_0}}\int_0^{T\eps^{-\gamma_0}} \nabla^\top \rho^1_t(X^1,\theta_0) \nabla \rho^1_t(X^1,\theta_0) dt\bigg) 
 M(\eps,\theta_0)^\top\phi(\eps).
\end{align*}
In view of \eqref{c1} 
$$
\frac{1}{T\eps^{-\gamma_0}}\int_0^{T\eps^{-\gamma_0}} \nabla^\top \rho^1_t(X^1,\theta_0) \nabla \rho^1_t(X^1,\theta_0) dt
\xrightarrow[\eps\to 0]{L_2(\Omega)} I(\theta_0;1),
$$
and hence \eqref{une} holds if $\phi(\eps)=\phi(\eps,\theta_0)$ satisfies \eqref{condphi}.  
It remains to show that \eqref{doux} holds for the same choice of $\phi(\eps,\theta_0)$. To this end,  
\begin{align*}
&
\E \frac  1\eps  \int_0^T\left(\int_0^1\int_0^\tau u^\top \phi(\eps)^\top  \nabla^2 \rho^\eps_t(X^\eps,\theta_0+s \phi(\eps) u) \phi(\eps) uds d\tau\right)^2
dt \le \\
&
\frac  1\eps \int_0^1 \int_0^T\E   \Big(  u^\top \phi(\eps)^\top  \nabla^2 \rho^\eps_t(X^\eps,\theta_0+s \phi(\eps) u ) \phi(\eps) u \Big)^2 
dtds.
\end{align*}
Under the condition \eqref{condphi}, $\|\phi(\eps,\theta_0)\|  = O\big( \eps^{\gamma_0/2}\log \eps^{-1}\big)$ and it suffices to check that 
$$
\eps^{2\gamma_0-1}\log^4\eps^{-1}    \int_0^T\E   \Big\|\nabla^2 \rho^\eps_t(X^\eps,\theta)\Big\|^2 
dt \xrightarrow[\eps\to 0]{}0,
$$
uniformly over $\{\theta:\|\theta_0-\theta\|\le \delta\}$ for all $\delta>0$ small enough.
Recall that  
$$
\nabla^2 \rho^\eps_t(X^\eps,\theta) = \int_0^t \nabla^2 g_\eps(t,t-s)\sigma_0 dB^{H_0}_s +\sqrt \eps
\int_0^t \nabla^2 g_\eps(t,t-s) dB_s.
$$
In view of  Lemma \ref{lem:s} 
\begin{align*}
&
\int_0^t \nabla^2 g_\eps(t,t-s)  dB^{H_0}_s =\\
&
\eps^{-\gamma} M(\eps, \theta) \int_0^t  \nabla^2 g_1 \big(t\eps^{-\gamma}, (t-s)\eps^{-\gamma}\big) dB^{H_0}_s M(\eps, \theta)^\top 
+ \\
&
 \eps^{-\gamma} \nu(\eps,\theta) \int_0^t  \frac{\partial}{\partial \sigma^2 } g_1 \big(t\eps^{-\gamma}, (t-s)\eps^{-\gamma}\big)dB^{H_0}_s \stackrel{d}=\\
&
\eps^{\gamma H_0-\gamma} M(\eps, \theta) \int_0^{t\eps^{-\gamma}}  \nabla^2 g_1 \big(t\eps^{-\gamma},  t\eps^{-\gamma}-s\big) dB^{H_0}_s  M(\eps, \theta)^\top 
+ \\
&
 \eps^{\gamma H_0-\gamma} \nu(\eps,\theta) \int_0^{t\eps^{-\gamma}}  \frac{\partial}{\partial \sigma^2 } g_1 \big(t\eps^{-\gamma},  t\eps^{-\gamma} -s \big) dB^{H_0}_s =: J_1(\eps,t\eps^{-\gamma})+J_2(\eps,t\eps^{-\gamma}).
\end{align*}
The first term satisfies 
\begin{align*}
&
\eps^{2\gamma_0-1}\log^4\eps^{-1}    \int_0^T\E   \Big\|J_1(\eps,t\eps^{-\gamma})\Big\|^2 dt =\\
&
\eps^{2\gamma_0-1}\log^4\eps^{-1} T \frac1{T\eps^{-\gamma}}   \int_0^{T\eps^{-\gamma}}\E   \Big\|J_1(\eps,t)\Big\|^2 dt  \le \\
&
\eps^{2\gamma_0-1}\log^4\eps^{-1}\eps^{2\gamma H_0-2\gamma} \big\|M(\eps, \theta) \big\|^4 T \frac1{T\eps^{-\gamma}}   \int_0^{T\eps^{-\gamma}}\E   
 \Big\| \int_0^t    \nabla^2 g_1 \big(t,  t-s\big) dB^{H_0}_s \Big\|^2 
 dt \le \\
&
 \eps^{2(\gamma_0-\gamma)+2(H_0-H)\gamma+\gamma} \log^8\eps^{-1}  T C\xrightarrow[\eps\to 0]{}0
\end{align*}
where the last bound is due to Lemma \ref{lem2} and the convergence holds uniformly over 
$\{\theta: \|\theta-\theta_0\|\le \delta\}$  for all sufficiently small $\delta>0$. 
Similarly, 
\begin{align*}
&
\eps^{2\gamma_0-1}\log^4\eps^{-1}    \int_0^T\E   \Big\|J_2(\eps,t\eps^{-\gamma})\Big\|^2 dt =\\
&
\eps^{2\gamma_0-1}\log^4\eps^{-1} T \frac 1 {T\eps^{-\gamma}}   \int_0^{T\eps^{-\gamma}} \E   \Big\|J_2(\eps,t )\Big\|^2 
dt  =\\
&
\eps^{2\gamma_0-1}\eps^{2\gamma H_0-2\gamma}\log^4\eps^{-1} \|\nu(\eps,\theta)\|^2  T \frac 1 {T\eps^{-\gamma}}  \int_0^{T\eps^{-\gamma}} \E   \Big\|
   \int_0^{t }  \frac{\partial}{\partial \sigma^2 } g_1 \big(t,  t  -s \big) dB^{H_0}_s
\Big\|^2 
dt \le \\
&
 \eps^{2(\gamma_0-\gamma)+2(H_0-H)\gamma+\gamma} \log^8\eps^{-1}   T C \xrightarrow[\eps\to 0]{}0,
\end{align*}
where the last inequality is due to Lemma \ref{lem1}.
Analogously, 
\begin{align*}
&
\sqrt \eps \int_0^t \nabla^2 g_\eps(t,t-s) dB_s = \\
&
\sqrt \eps  \eps^{-\gamma} M(\eps, \theta)
\Big(\int_0^t \nabla^2 g_1 \big(t\eps^{-\gamma}, (t-s)\eps^{-\gamma}\big)  dB_s\Big) M(\eps, \theta)^\top
+ \\
&
\sqrt \eps \eps^{-\gamma} \nu(\eps,\theta)
\int_0^t \frac{\partial}{\partial \sigma^2 } g_1 \big(t\eps^{-\gamma}, (t-s)\eps^{-\gamma}\big)dB_s \stackrel{d}=\\
&
  \eps^{(1-\gamma)/2}   M(\eps, \theta)
\Big(\int_0^{t\eps^{-\gamma}} \nabla^2 g_1 \big(t\eps^{-\gamma},  t\eps^{-\gamma}-s\big)  dB_s\Big) M(\eps, \theta)^\top
+ \\
&
 \eps^{(1-\gamma)/2}  \nu(\eps,\theta)
\int_0^{t\eps^{-\gamma}} \frac{\partial}{\partial \sigma^2 } g_1 \big(t\eps^{-\gamma},  t\eps^{-\gamma}-s\big)dB_s =:
J_1(\eps,t\eps^{-\gamma}) + J_2(\eps,t\eps^{-\gamma}).
\end{align*}
The first term vanishes asymptotically as $\eps\to 0$,
\begin{align*}
&
\eps^{2\gamma_0-1}\log^4\eps^{-1}    \int_0^T\E   \Big\|J_1(\eps,t\eps^{-\gamma})\Big\|^2 
dt  = \\
&
\eps^{2\gamma_0-1}\log^4\eps^{-1} T  \frac{1}{T\eps^{-\gamma}} \int_0^{T\eps^{-\gamma}}\E   \Big\|J_1(\eps,t)\Big\|^2 
dt \le \\
&
\eps^{2\gamma_0-1} \eps^{1-\gamma}\log^4\eps^{-1} \|M(\eps, \theta)\|^4 T  \frac{1}{T\eps^{-\gamma}} \int_0^{T\eps^{-\gamma}}\E   \Big\|
\int_0^t \nabla^2 g_1 \big(t,  t-s\big)  dB_s   
\Big\|^2 
dt \le \\
&
\eps^{\gamma_0+(\gamma_0-\gamma)}  \log^4\eps^{-1} \|M(\eps, \theta)\|^4 T  C \xrightarrow[\eps\to 0]{}0,
\end{align*}
where the last inequality holds by Lemma \ref{lem2}. Similarly, 
$$
\eps^{2\gamma_0-1}\log^4\eps^{-1}    \int_0^T\E   \Big\|J_2(\eps,t\eps^{-\gamma})\Big\|^2 
dt\xrightarrow[\eps\to 0]{}0.
$$
This verifies \eqref{doux} and completes the proof.

\subsection{Proof of Corollary \ref{cor1}}
For brevity denote the matrices in \eqref{condphi} by $M:= M(\eps,\theta_0)$ and $I:=I(\theta_0;1)$. 
Consider the Cholesky decomposition 
$$
M I M^\top=LL^\top,
$$
where $L$ is the unique lower triangular matrix with positive diagonal entries. A simple calculation shows that 
$$
L = \begin{pmatrix}
  m(\eps)  \sqrt{I_{22}}  & 0 \\
   \sqrt{I_{22}}  & 
\displaystyle\frac 1 {m(\eps)} \sqrt{ I_{11}    - I_{12}^2/I_{22} }   
\end{pmatrix} \big(1+o(1)\big), \quad \eps \to 0
$$
where $m(\eps)= 2\sigma_0^2  \log\eps^{-1/(2H_0-1)}$. Hence the matrix 
$$
\phi(\eps, \theta_0)^\top  = \eps^{1/(4H_0-2)}\frac 1 {\sqrt{T}} L^{-1}
$$
satisfies condition \eqref{condphi}. The assertion (2) of Corollary \ref{cor1} is obtained by applying Theorem \ref{thm:HLC}
to loss functions constant in the first variable. The assertion (1) is proved similarly, using the upper 
triangular Cholesky decomposition. 

\subsection{Proof of Theorem \ref{thm3}}
For a fixed $\sigma_0^2$, the LAN property of the one dimensional family 
$\big(\P^\eps_{(H,\sigma_0^2)}\big)_{H\in (3/4,1)}$
is obtained by considering the likelihood ratio \eqref{LReps} with diagonal $\phi(\eps,\theta_0)$ and 
$u$ restricted to the line $\{u_1 e_1: u_1\in \Real\}$ where $e_1 = (1,0)^\top$. For the vectors from this subspace, 
the limit \eqref{une} holds if, cf. \eqref{condphi},
$$
\eps^{-1/(2H_0-1)} e_1^\top \phi(\eps,\theta_0)^\top M(\eps,\theta_0) T I(\theta_0;1)M(\eps,\theta_0)^\top\phi(\eps,\theta_0)e_1\xrightarrow[\eps\to 0]{}1.
$$
For diagonal $\phi(\eps, \theta_0)$, this convergence is true if
$$
\phi_{11}(\eps, \theta_0) = \eps^{1/(4H_0-2)} \frac 1 {2\sigma_0^2  \log\eps^{-1/(2H_0-1)}}\frac 1{\sqrt{T I_{22}(\theta_0;1)}}
$$ 
which is the scaling claimed in \eqref{rateH}. The property \eqref{doux} continues to hold as before. This proves assertion (1) of 
Theorem \ref{thm3}. Assertion (2) is proved analogously, by restricting $u$ to the subspace $\{u=u_2e_2: u_2\in \Real\}$
with $e_2 = (0,1)^\top$.


\def\cprime{$'$} \def\cprime{$'$} \def\cydot{\leavevmode\raise.4ex\hbox{.}}
  \def\cprime{$'$} \def\cprime{$'$} \def\cprime{$'$}

\end{document}